\documentclass[11pt,a4paper]{article}
\usepackage{amsmath,amssymb,amsthm}
\usepackage[a4paper,margin=27mm]{geometry}
\usepackage[hidelinks]{hyperref}
\usepackage{orcidlink}
\newtheorem{theorem}{Theorem}
\newtheorem{proposition}[theorem]{Proposition}
\newtheorem{lemma}[theorem]{Lemma}
\newtheorem{corollary}[theorem]{Corollary}
\newtheorem{conjecture}[theorem]{Conjecture}
\theoremstyle{definition}

\ifdefined\pdfinfoomitdate\pdfinfoomitdate=1\fi
\ifdefined\pdfsuppressptexinfo\pdfsuppressptexinfo=15\fi
\ifdefined\pdftrailerid\pdftrailerid{}\fi
\ifdefined{/Creator () /Producer ()}\fi

\usepackage[T1]{fontenc}
\usepackage{lmodern,microtype,mathtools,mathrsfs}
\usepackage{booktabs,array,graphicx,tabularx,placeins,flafter}
\usepackage{enumitem,algorithm,float}
\usepackage[noend]{algpseudocode}
\usepackage[font=small,labelfont=bf]{caption}
\graphicspath{{figures/}}
\newcommand{\Z}{\mathbb Z}
\newcommand{\R}{\mathbb R}

\newcommand{\conv}{\operatorname{conv}}
\newcommand{\relint}{\operatorname{relint}}

\newcommand{\vol}{\operatorname{Vol}}
\newcommand{\Av}{\operatorname{Av}}
\newcommand{\cC}{\mathcal C}
\newcommand{\cL}{\mathcal L}
\newcommand{\cQ}{\mathcal Q}

\newcommand{\cstar}{c_*}
\newcommand{\ystar}{y_*}

\title{On the real-rootedness of Ehrhart $h^*$-polynomials\\of $(132,213)$-avoiding permutahedra}
\author{Pedro M. M. de Castro\,\orcidlink{0000-0002-9470-7458}\\
\small Centro de Inform\'atica, Universidade Federal de Pernambuco\\
\small Recife, Pernambuco, Brazil\\
\small \texttt{pmmc@cin.ufpe.br}}
\date{October 28, 2026}
\hypersetup{
 pdftitle={On the real-rootedness of Ehrhart h-star-polynomials of (132,213)-avoiding permutahedra},
 pdfauthor={Pedro M. M. de Castro},
 pdfcreator={},pdfproducer={},
 pdfkeywords={Ehrhart theory, pattern avoidance, poset permutahedron, Newton polytope of the discriminant, tournament score sequence, refined Eulerian polynomial, interlacing, real-rooted polynomial, exact certificate}
}
\begin{document}
\maketitle

\begin{abstract}
Let $P_d$ be the convex hull of the permutations of $\{1,\ldots,d\}$ that avoid the patterns $132$ and $213$, and let $H_d(t)$ be its Ehrhart $h^*$-polynomial. We prove that all zeros of $H_d(t)$ are real and negative for $3\le d\le1000$ and for every $d\ge2^{62}$; for $d\le2$ the polynomial is constant. Expanding $H_d(t)$ in refined Eulerian polynomials does not suffice, since a coordinate is already negative at $d=7$. Let $A_{d,i}(t)$ enumerate all permutations of $\{1,\ldots,d\}$ with first entry $i$ by number of descents. For $c\ge0$, we prove that every nonnegative linear combination of $A_{d,1},\ldots,A_{d,d}$ and the adjacent differences $A_{d,r+1}-cA_{d,r}$, $2\le r\le d-2$, is real-rooted for all $d\ge4$ simultaneously if and only if $c\le2-\sqrt3$. This reduces real-rootedness to the positivity of explicit weighted tails of the coordinates. An exact certificate checks these inequalities for $d\le1000$, and coefficient estimates based on Darroch's theorem on the mode of a sum of Bernoulli variables prove them for $d\ge2^{62}$; the range $1001\le d<2^{62}$ remains open. Along the way we identify $P_d$, up to lattice equivalence, with twice a chain-poset permutahedron and with the Newton polytope of a discriminant, which transfers the known cube structure and volume of these models to $P_d$. We also show that the coefficient of $t^{d-1}$ in $H_d(t)$ counts strong tournament score sequences for $d\ge3$, confirming a conjecture of Rehberg in this range.
\end{abstract}

\medskip
\noindent\textbf{AI-use disclosure.}
OpenAI's ChatGPT and Codex, as well as Anthropic's Claude, were used as research tools to assist the author with literature retrieval, mathematical exploration, symbolic and numerical checks, and draft preparation and revision. The author formulated the research questions and mathematical framework, developed and refined the results and arguments with this assistance, and critically reviewed the manuscript throughout. Responsibility for the mathematical statements and the final manuscript rests entirely with the author.

\medskip
\noindent\textbf{Keywords.} Ehrhart theory; pattern avoidance; refined Eulerian polynomials; real-rootedness; tournament score sequences.

\smallskip
\noindent\textbf{2020 Mathematics Subject Classification.} 05A15, 52B20, 26C10; 05A05, 05A16, 05C20, 52B12.

\section{Introduction}

The face lattice of a polytope and the zeros of its Ehrhart numerator record different information. Faces describe how the boundary fits together, while the numerator of the Ehrhart series encodes the lattice-point counts of all integral dilates. When this numerator has nonnegative coefficients and only real zeros, its coefficients form a log-concave and unimodal sequence, and the unimodality, log-concavity, and real-rootedness of $h^*$-polynomials are recurring questions in Ehrhart theory \cite{FerroniHigashitani2024}. The polytopes studied here combine these aspects in a simple form. Their face lattice is that of a cube, for $d\ge3$ the top coefficient of their Ehrhart numerator counts score sequences of strong tournaments, and the zeros of the numerator lead to a signed extension of a familiar compatible family of Eulerian polynomials.

A permutation \emph{avoids} a pattern when no subsequence has the same relative order as that pattern; for example, the entries $2,4,3$ of $2413$ form a copy of $132$. Write $\Av_d(\Pi)$ for the set of permutations of $[d]=\{1,\ldots,d\}$ that avoid every pattern in $\Pi$. Davis and Sagan \cite{DavisSagan2018} introduced the polytopes
\[
 P_d(\Pi)=\conv\{(\pi_1,\ldots,\pi_d):\pi\in\Av_d(\Pi)\},
\]
which they call $\Pi$-avoiding permutohedra. They asked whether $P_d(132,213)$ is combinatorially a $(d-1)$-cube with normalized volume $2^{d-1}d^{d-3}$ \cite[Conjecture~3.19]{DavisSagan2018}, and, more generally, for which $\Pi$ the $h^*$-polynomial of $P_d(\Pi)$ is unimodal \cite[Question~3.13]{DavisSagan2018}. Cubicality concerns face incidences, while normalized volume is measured in the affine lattice so that a unimodular simplex has volume one. For $\Pi=\{132,312\}$, Davis and Sagan derive unimodality \cite[Corollary~3.12]{DavisSagan2018} from a theorem of Beck, Jochemko, and McCullough: the $h^*$-polynomial of every lattice zonotope is real-rooted and unimodal \cite[Theorem~1.2]{BeckJochemkoMcCullough2019}. The permutations that avoid $132$ and $213$ consist of increasing blocks whose values decrease from one block to the next. This rigid structure makes the family suitable for a direct comparison of geometry and enumeration, but it does not produce zonotopes: already $P_3(132,213)$ is a quadrilateral that is not a parallelogram.

Write $P_d=P_d(132,213)$. Four descriptions of the same lattice polytope organize the paper. The blocks label the vertices by their cuts. Prefix deficits turn the facet inequalities into a path-Laplacian system, in which each cut is a zero of a piecewise parabolic profile. Two affine lattice maps then identify $P_d$ with twice a translated poset permutahedron of a chain and with the Newton polytope of the discriminant of a polynomial of degree $d$. Through the discriminant model, the cube structure and the volume become classical facts; through the chain model, we reach tournament scores and, using the symmetry of the permutahedron, the Ehrhart series.

Define $H_d(t)$ by
\[
 \sum_{m\ge0}|mP_d\cap\Z^d|t^m=\frac{H_d(t)}{(1-t)^d}.
\]
This is the usual Ehrhart $h^*$-polynomial of the $(d-1)$-dimensional lattice polytope $P_d$ \cite{BeckRobins2015}. Table~\ref{tab:small-numerators} displays the first nonconstant cases. The coefficient sum is the normalized volume, and the final coefficient counts interior lattice points. All four polynomials have only real negative zeros, and $H_3(t)=(1+t)^2$ has a repeated zero.

\begin{table}[ht]
\centering
\caption{Initial Ehrhart numerators, obtained from the generating series and the refined transform of Section~\ref{sec:eulerian}; the same vectors appear in the polymake computations of Tamayo Jim\'enez \cite[Appendix~B.1]{Tamayo2017}. The last two columns also follow from Corollary~\ref{cor:volume} and from the interior-point bijection of Section~\ref{sec:tournaments}. The cases $d=1,2$ have $H_d(t)=1$.}
\label{tab:small-numerators}
\small
\begin{tabular}{@{}clrr@{}}
\toprule
$d$ & Coefficients of $H_d(t)$, in ascending order & Volume & Interior points\\
\midrule
3 & $(1,2,1)$ & 4 & 1\\
4 & $(1,12,18,1)$ & 32 & 1\\
5 & $(1,54,236,106,3)$ & 400 & 3\\
6 & $(1,241,2529,3438,696,7)$ & 6912 & 7\\
\bottomrule
\end{tabular}
\end{table}

Write $S_d$ for the set of permutations of $[d]$. A \emph{descent} of $\pi=(\pi_1,\ldots,\pi_d)\in S_d$ is a position $k$ with $1\le k<d$ and $\pi_k>\pi_{k+1}$, and $\operatorname{des}(\pi)$ denotes the number of descents. For $1\le i\le d$, the coefficient of $t^j$ in the refined Eulerian polynomial $A_{d,i}(t)$ counts permutations with first entry $i$ and exactly $j$ descents.

The zeros pose the main difficulty. Burnside averaging over the symmetric group expresses the lattice-point count of the chain model through the exponential of a generating series whose coefficients are explicit. A refined Worpitzky transform then gives an expansion $H_d(t)=\sum_{j=0}^{d-1}a_{d,j}A_{d,j+1}(t)$. The polynomials $A_{d,i}$ form a \emph{compatible} family, meaning that every nonzero nonnegative linear combination of them is real-rooted. Thus nonnegative coordinates would suffice, but $a_{7,1}=-1/15$ already prevents that argument. We adjoin the differences $A_{d,r+1}-cA_{d,r}$ of adjacent basis elements, with $c\ge0$, and prove that the enlarged family remains compatible for all $d\ge4$ simultaneously exactly when $0\le c\le c_*=2-\sqrt3$. In a basis drawn from the enlarged family, the coordinates of $H_d(t)$ are weighted tails of the original coordinates, and their positivity is the common target of the finite and analytic parts of the paper.

\subsection*{Main results}

We first fix the lattice realization. The cube structure and the volume in the next theorem are known for the two models, and the lattice equivalences transfer them to $P_d$; the equivalences also fix the lattice normalization used in the Ehrhart calculation.

\begin{theorem}\label{thm:intro-geometry}
For every $d\ge1$, the affine map $\Psi_d(x)_i=d+i-1-x_i$ is a lattice equivalence
\[
 P_d(132,213)\cong_{\Z}2(\Pi_{\cC_d}-\mathbf1),
\]
where $\cC_d$ is the chain $1<\cdots<d$, $\Pi_{\cC_d}$ is its poset permutahedron, and $\mathbf1=(1,\ldots,1)$. For $d\ge2$, the polytope $P_d(132,213)$ is also lattice equivalent to the Newton polytope of the discriminant of a polynomial of degree $d$. Consequently $P_d(132,213)$ is combinatorially a $(d-1)$-cube and
\[
 \vol_{\mathrm{norm}}P_d(132,213)=2^{d-1}d^{d-3}.
\]
\end{theorem}
The proof of Theorem~\ref{thm:intro-geometry} is given in Sections~\ref{sec:geometry} and~\ref{sec:models}. Theorem~\ref{thm:cube-face-lattice} gives a self-contained proof of the cube structure in deficit coordinates, Theorem~\ref{thm:lattice-map} and Proposition~\ref{prop:discriminant} prove the two lattice equivalences, and Corollary~\ref{cor:volume} gives the normalized volume. For $d=1$, the polytope is a point. The cube structure was proved for $P_d(132,213)$ by Tamayo Jim\'enez \cite[Proposition~3.5.2]{Tamayo2017} and for the discriminant polytope by Gelfand, Kapranov, and Zelevinsky, as recalled below. Through the equivalences, the known volumes of the two models settle the second part of the question of Davis and Sagan.

A \emph{tournament} is an orientation of a complete graph. It is \emph{strong} if every vertex can reach every other by a directed path. Its \emph{score sequence} lists the outdegrees of the vertices in nondecreasing order, and a strong score sequence is one realized by a strong tournament. We count each score sequence once, regardless of the number of tournaments realizing it. Landau's inequalities and their strict version relate these sequences to the chain model \cite{Landau1953,HararyMoser1966}.

\begin{theorem}\label{thm:intro-tournament}
For every $d\ge3$, the coefficient of $t^{d-1}$ in $H_d(t)$ is the number of strong tournament score sequences on $d$ vertices. The same number is the leading coefficient of the rational Ehrhart numerator $h^*_{\Z}(\Pi_{\cC_d};t)$, which has degree $2d-2$.
\end{theorem}
The proof of Theorem~\ref{thm:intro-tournament} appears in Section~\ref{sec:tournaments}, where the rational numerator is defined. A shift of the ordered coordinates turns the strict prefix inequalities of the interior into the strict form of Landau's criterion. The second assertion is Rehberg's conjecture \cite[Conjecture~6.36]{Rehberg2025} for $d\ge3$. It also holds for $d=1$, while for $d=2$ the rational numerator is $1+t$ and no tournament on two vertices is strong.

The main new results concern the zeros of $H_d(t)$. A single criterion, the positivity of the weighted tails, proves real-rootedness in a finite initial range and in an explicit eventual range.

\begin{theorem}\label{thm:intro-finite}
For every $1\le d\le1000$, all zeros of $H_d(t)$ are real and negative. Here $H_1(t)=H_2(t)=1$ and $H_3(t)=(1+t)^2$; the assertion is vacuous for $d=1,2$.
\end{theorem}
Section~\ref{sec:finite} proves Theorem~\ref{thm:intro-finite} using the criterion of Section~\ref{sec:repair}. All arithmetic comparisons are exact and can be reconstructed from the displayed pseudocode and the public data.

\begin{theorem}\label{thm:intro-eventual}
For every integer $d\ge d_0=2^{62}$, all zeros of $H_d(t)$ are real and negative.
\end{theorem}
Sections~\ref{sec:reverse} and~\ref{sec:wlc} prove Theorem~\ref{thm:intro-eventual}, with the final argument assembled in Theorem~\ref{thm:WLC}. It separates the tail indices into two ranges and establishes positivity in both. A polynomial with only real negative zeros has positive, log-concave, and therefore unimodal coefficients. Theorems~\ref{thm:intro-finite} and~\ref{thm:intro-eventual} thus show that the $h^*$-polynomial of $P_d(132,213)$ is unimodal for $d\le1000$ and for $d\ge2^{62}$, which gives a partial answer to Question~3.13 of Davis and Sagan for $\Pi=\{132,213\}$.

For the small tail indices, a Cauchy estimate compares the coordinates with the value of their generating polynomial at $c_*$. For the remaining indices, it suffices to show that the corresponding coordinates are positive, and we compare them with the coefficients of an explicit product. After normalization, each factor of this product is the generating polynomial of a Bernoulli variable, as in Pitman's probabilistic reading of real-rooted polynomials \cite{Pitman1997}, and we choose the normalization so that the target coefficient sits at the mean. Darroch's theorem \cite[Theorem~4]{Darroch1964}, in the form stated by Pitman \cite[p.~284]{Pitman1997}, then shows that this coefficient is the largest point probability, since for a sum of independent Bernoulli variables every mode lies at distance less than one from the mean. An elementary Fourier estimate controls the change of adjacent probabilities, and first moments of a coefficient majorant control the remaining error. The threshold $2^{62}$ is the point where the resulting error bound $5\cdot2^{28}/\sqrt d$ drops to $5/8$; the range $1001\le d<d_0$ remains open.

\subsection*{Related work}

Gelfand, Kapranov, and Zelevinsky proved that the Newton polytope of the discriminant of a univariate polynomial of degree $d$ is combinatorially a $(d-1)$-cube and described its vertices through subdivisions of the segment $[0,d]$ \cite{GKZ1990}, \cite[Chapter~12, Theorem~2.2]{GKZ1994}. Orevkov computed its volume by a triangulation indexed by permutations \cite{Orevkov1999}. Black, Rehberg, and Sanyal introduced poset permutahedra, observed that twice the poset permutahedron of a chain is unimodularly equivalent to this Newton polytope, and obtained a volume formula for all poset permutahedra \cite{BlackRehbergSanyal2025}. For the pattern-avoiding model itself, Tamayo Jim\'enez proved the block description of the vertices, the paired facet inequalities, and cubicality \cite{Tamayo2017}, and computed the $h^*$-vectors for $d\le6$ \cite[Appendix~B.1]{Tamayo2017}. The chain model is also a partitioned weight polytope of type $A$, whose facets were described by Horiguchi, Masuda, Shareshian, and Song \cite[Corollary~4.3]{HoriguchiEtAl2024}, and its cube structure is a case of the theorem of Burrull, Gui, and Hu on strongly dominant weight polytopes \cite[Theorem~2.9]{BurrullGuiHu2024}: in type $A_{d-1}$, the centered weight $(1,\ldots,d)-\frac{d+1}{2}\mathbf1$ lies strictly inside the chamber of increasing coordinates, and the intersection of its orbit polytope with that chamber is the centered chain model. The Ehrhart numerator depends on the lattice realization, and our geometric contribution is limited to the explicit lattice equivalences of Theorem~\ref{thm:intro-geometry} together with a self-contained proof in deficit coordinates.

Black, Rehberg, and Sanyal identify score vectors of poset tournaments with lattice points of poset permutahedra translated by minus the all-ones vector \cite[Theorem~1.4]{BlackRehbergSanyal2025}; for a chain these are Landau's score sequences. Rehberg characterized strong score sequences as the lattice points of the translated chain permutahedron that lie on none of its filter facets, and conjectured that their number is the leading coefficient of its rational Ehrhart numerator \cite[Corollary~6.35 and Conjecture~6.36]{Rehberg2025}. Claesson, Dukes, Frankl\'in, and Stef\'ansson \cite{ClaessonEtAl2023} and Stockmeyer \cite{Stockmeyer2023} give counting formulas and algorithms for score sequences, and the asymptotic enumeration is studied by Kolesnik \cite{Kolesnik2023} and by Bassan, Donderwinkel, and Kolesnik \cite{BassanDonderwinkelKolesnik2026}.

The main contribution is the compatible family of adjacent differences and its use in proving real-rootedness. The refined Eulerian polynomials used here occur in the transformation of $h$-vectors under barycentric subdivision \cite{BrentiWelker2008}, and we adopt the first-letter indexing of Nevo, Petersen, and Tenner, who call them restricted Eulerian polynomials \cite{NevoPetersenTenner2011}. Every nonnegative combination of them is real-rooted \cite[Remark~1]{BrentiWelker2008}, and they form an interlacing sequence \cite[Example~2.3]{Athanasiadis2025}. Beck, Jochemko, and McCullough showed that the $h^*$-polynomial of every lattice zonotope is a nonnegative combination of these polynomials, which gives its real-rootedness \cite[Theorems~1.2 and~1.4]{BeckJochemkoMcCullough2019}. The convex hull of these zonotopal $h^*$-polynomials is the translated simplicial cone $A_{d,1}+\sum_{i=2}^d\mathbb R_{\ge0}A_{d,i}$ \cite[Theorem~1.3]{BeckJochemkoMcCullough2019}. Our expansion of $H_d(t)$ uses the same coordinates, but $P_d$ is not a zonotope for $d\ge3$, and $H_7(t)$ lies outside that cone. Real-rooted polynomials with negative refined Eulerian coordinates have been treated before by other means. Zhang proved, extending a result of Br\"and\'en, that $A_n(t)+k\,tA_{n-2}(t)$ is real-rooted for $n>3$ and $k\ge-n$, where $A_n(t)=\sum_{\pi\in S_n}t^{\operatorname{des}(\pi)}$ \cite[Theorems~1.2 and~1.3]{Zhang2015}; a direct computation shows that this polynomial has negative coordinates in the refined basis already for $k=-1$ and $n=11$. In his proof, such polynomials are written as sums $\sum_i(a_it+b_i)f_i$ over an interlacing sequence. In Section~\ref{sec:repair} we instead adjoin adjacent differences with a constant parameter to the refined family, and the constant $2-\sqrt3$ is optimal for this construction uniformly in $d$; other compatible families may give stronger criteria. Eulerian transformations also occur in \cite{BrandenJochemko2022,Athanasiadis2025}; their operators and positivity hypotheses differ from the signed expansion considered here. Compatibility and interlacing provide the relevant tools \cite{ChudnovskySeymour2007,SavageVisontai2015}, and we fix the orientation of interlacing in Section~\ref{sec:repair}.

\subsection*{Organization of the paper}

Section~\ref{sec:geometry} proves the facet description and the face lattice in deficit coordinates. Section~\ref{sec:models} establishes the two lattice models and the volume, and Section~\ref{sec:tournaments} interprets the top coefficient through strong tournaments. Sections~\ref{sec:eulerian} and~\ref{sec:repair} derive the refined coordinates and the criterion based on adjacent differences. Section~\ref{sec:finite} verifies the finite range. Sections~\ref{sec:reverse} and~\ref{sec:wlc} prove eventual positivity, covering the lower tail indices with a forward Cauchy estimate and the remaining indices with estimates for the reversed coordinates. Section~\ref{sec:frontiers} discusses the open range and the coordinate inequalities that would settle it.
\section{Pattern avoidance, deficits, and the cube}\label{sec:geometry}

This section proves the facet description and the cubical face lattice of
$P_d$ in coordinates adapted to its vertices. Tamayo Jim\'enez obtained the
block description of the vertices, the paired facet inequalities, and
cubicality \cite[Propositions~3.5.1--3.5.2]{Tamayo2017}, and the
cube structure also follows from the discriminant model of
Section~\ref{sec:models}. We nevertheless give a short self-contained proof:
the lattice equivalences of Section~\ref{sec:models} start from these
inequalities, and the same coordinates attach to every vertex an explicit
integer profile. The original coordinates list the entries of a permutation,
which is natural for describing vertices but hides the facets. Prefix
deficits measure how far each initial sum lies below its largest possible
value. In these coordinates the two facet choices at every position become a
zero condition and a path-Laplacian condition, and the cuts of a
reverse-layered permutation appear as zeros of a piecewise parabolic profile.

All vertices of $P_d$ lie in the affine lattice
\[
 \cL_d=\left\{x\in\Z^d:\sum_{i=1}^d x_i=\frac{d(d+1)}2\right\},
\]
and normalized volume is measured with respect to its difference lattice
\[
 \mathsf A_{d-1}=\left\{z\in\Z^d:\sum_{i=1}^d z_i=0\right\}.
\]
To label the vertices canonically, we first describe the avoidance class
through cuts of reverse-layered permutations. The \emph{skew sum}
$\sigma\ominus\tau$ places a value-shifted copy of $\sigma$ before $\tau$. A
\emph{reverse-layered permutation} is a skew sum of increasing permutations.

\begin{lemma}\label{lem:reverse-layered}
The permutations in $\Av_d(132,213)$ are exactly the skew sums of increasing blocks. Hence they are indexed by subsets of $[d-1]$ and there are $2^{d-1}$ of them.
\end{lemma}

\begin{proof}
Suppose that the maximum entry $d$ occurs in position $b$. The entries before $d$ are increasing, since a descent followed by $d$ would form $213$. Every entry before $d$ exceeds every entry after it, since an increasing cross-pair together with $d$ would form $132$. Thus the first $b$ entries are the $b$ largest values in increasing order. Induction applies to the suffix. Conversely, three entries chosen from a skew sum of increasing blocks have pattern $123$, $231$, $312$, or $321$. A cut after each selected position in $[d-1]$ specifies the blocks uniquely.
\end{proof}

For $1\le k<d$, set
\[
 S_k(x)=\sum_{i=1}^k x_i,
 \qquad
 M_k=\frac{k(2d-k+1)}2.
\]
Consider the paired inequalities
\begin{align}
 A_k&: x_k-x_{k+1}\ge-1,\label{eq:A-facet}\\
 B_k&: S_k(x)\le M_k.\label{eq:B-facet}
\end{align}
Define prefix deficits
\[
 u_k=M_k-S_k(x),\qquad u_0=u_d=0.
\]
The inverse transformation is
\begin{equation}\label{eq:inverse-deficit}
 x_k=d-k+1+u_{k-1}-u_k.
\end{equation}
Let $L$ be the Dirichlet path Laplacian on the coordinates
$1,\ldots,d-1$, with the boundary values $u_0=u_d=0$:
\[
 (Lu)_k=2u_k-u_{k-1}-u_{k+1}\qquad(1\le k<d).
\]
Then \eqref{eq:A-facet} and \eqref{eq:B-facet} become
\begin{equation}\label{eq:Qd}
 \cQ_d=\left\{u\in\R^{d-1}:u_k\ge0,\ (Lu)_k\le2\text{ for }1\le k<d\right\}.
\end{equation}

An inequality is \emph{tight} at a point if it holds with equality there.
In deficit coordinates, $B_k$ is tight exactly when $u_k=0$, and $A_k$ is tight
exactly when $(Lu)_k=2$. These two alternatives give the facet choices a
direct geometric meaning. Choosing one condition at each position will label the vertices by
subsets of $[d-1]$. For a set $C\subseteq[d-1]$ of cuts, adjoin the boundary
points $0$ and $d$ and define the profile
\[
 u^C_k=(k-p)(p'-k)\qquad(p\le k\le p'),
\]
where $p<p'$ are consecutive points of $C\cup\{0,d\}$. It vanishes exactly at
the cuts and satisfies $(Lu^C)_k=2$ at every other position. On a block
$p<k\le p'$, formula \eqref{eq:inverse-deficit} gives $x_k=d+k-p-p'$, so the
inverse deficit map sends $u^C$ to the reverse-layered permutation with cut set
$C$; Lemma~\ref{lem:signatures} below shows that the profiles $u^C$ are exactly
the vertices of $\cQ_d$. We call this correspondence between cut sets,
profiles, and permutations the \emph{vertex dictionary}. For example,
\[
 \{1,3\}\longmapsto u^{\{1,3\}}=(0,1,0)
 \longmapsto 4231.
\]

Figure~\ref{fig:low-dimensional-deficits} compares the first two nontrivial exact realizations. In $\cQ_3$, the two line styles identify the paired facet families of \eqref{eq:Qd}. In the projected realization of $\cQ_4$, edges with the same style toggle the same cut, so the cubical adjacency is visible together with the metric deformation of the cube.

\begin{figure}[t]
\centering
\includegraphics[width=.86\textwidth]{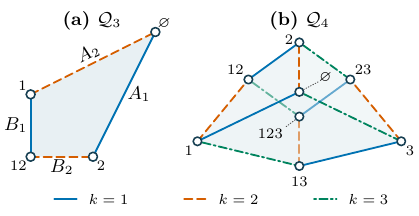}
\caption{Deficit geometry in affine dimensions two and three. Vertex labels are cut sets, with braces suppressed. In (a), the style indexed by $k$ identifies the facet pair $A_k,B_k$ of $\cQ_3$. In (b), it identifies the cut $k$ toggled along an edge of $\cQ_4$. The drawing in (b) is an orthographic projection; faint boundary edges lie inside its projected outline.}
\label{fig:low-dimensional-deficits}
\end{figure}

The cubical adjacency does not determine the metric realization. Set
$\widehat u_k=u_k/[k(d-k)]$, and write $\widehat{\cQ}_d$ for the image of
$\cQ_d$ under this diagonal normalization. Figure~\ref{fig:deficit-deformation}
compares the cube graphs with projections of the normalized deficit realizations for permutation sizes $d=4$ and $d=5$. Their affine dimensions are $3$ and $4$, respectively. The diagonal normalization is used to compare shapes; the lattice and volume calculations use the original $u$-coordinates.

\begin{figure}[t]
\centering
\includegraphics[width=.96\textwidth]{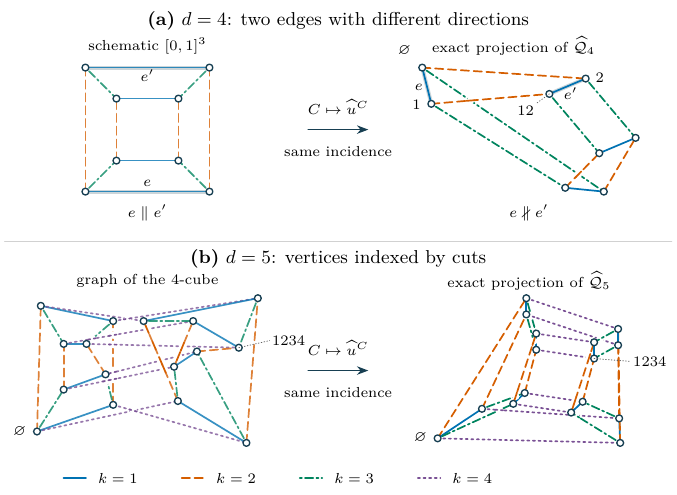}
\caption{The cut correspondence preserves incidence while changing shape. In (a), we compare the cube graph with a projection of the three-dimensional polytope $\widehat{\cQ}_4$. The opposite edges $e=(\varnothing,\{1\})$ and $e'=(\{2\},\{1,2\})$ have nonparallel images in deficit coordinates. In (b), we show all $16$ vertices and $32$ edges for the four-dimensional polytope $\widehat{\cQ}_5$. A style indexed by $k$ records the toggled cut. The drawings on the left are schematic graph layouts; those on the right are linear projections of the normalized vertices. Crossings away from marked vertices carry no incidence information.}
\label{fig:deficit-deformation}
\end{figure}

The affine obstruction in Figure~\ref{fig:deficit-deformation}(a) can also be read directly from the coordinates:
\[
 \widehat u^{\{1\}}-\widehat u^{\varnothing}=(-1,-1/2,-1/3),
 \qquad
 \widehat u^{\{1,2\}}-\widehat u^{\{2\}}=(-1/3,0,0).
\]
These are opposite edges of the face on which cut $3$ is absent. Their unequal directions certify a quadrilateral face that is not a parallelogram, and hence an obstruction to affine equivalence with the standard cube. The same face exists for every $d\ge3$: when positions $1$ and $2$ are free and no other position is a cut, the edges $u^{\{1\}}-u^{\varnothing}=-(d-1,d-2,\ldots,1)$ and $u^{\{1,2\}}-u^{\{2\}}=-(1,0,\ldots,0)$ are opposite and not parallel. Since zonotopes have only centrally symmetric faces, $P_d$ is not a zonotope for $d\ge3$.

The polytope $\cQ_d$ is bounded and full-dimensional. Every $u\in\cQ_d$ satisfies
\[
 0\le u_k\le k(d-k).
\]
Indeed, let $h_k=k(d-k)$, with $h_0=h_d=0$. Direct substitution gives
\[
 (Lh)_k=2h_k-h_{k-1}-h_{k+1}=2\qquad(1\le k<d).
\]
Put $w=h-u$. Since $(Lu)_k\le2$ for every $1\le k<d$, we have
$(Lw)_k\ge0$ at every interior index, and $w_0=w_d=0$.
If $w_k<0$ were an interior minimum, then
$2w_k-w_{k-1}-w_{k+1}\le0$. Equality must therefore hold, and both
neighbors must equal $w_k$. Repeating this argument reaches a boundary
index, where $w=0$, a contradiction. Thus $w_k\ge0$ for every
$1\le k<d$. Full dimension follows because $\varepsilon h$ satisfies
every defining inequality strictly whenever $0<\varepsilon<1$.

\begin{lemma}\label{lem:signatures}
At every point of $\cQ_d$, at most one inequality in each pair $u_k\ge0$ and $(Lu)_k\le2$ is tight. At every vertex, exactly one inequality in each pair is tight. For every choice of one inequality from each pair, the profile $u^C$, where $C$ is the set of positions at which $u_k\ge0$ is chosen, is the unique point of $\cQ_d$ at which all chosen inequalities are tight.
\end{lemma}

\begin{proof}
If $u_k=0$, nonnegativity of its neighbors gives $(Lu)_k\le0$, so the Laplacian inequality is strict. Thus at most one inequality in each pair is tight at a feasible point. At a vertex of a full-dimensional polytope in $\R^{d-1}$, the normal vectors of the tight inequalities span $\R^{d-1}$. Hence at least $d-1$ inequalities are tight, and exactly one from each pair is tight.

For a given choice, let $C$ be the set of positions at which $u_k\ge0$ is chosen. The profile $u^C$ is feasible: at a cut, $u^C_k=0$ and the paired inequality is strict, while at a noncut, $u^C_k>0$ and $(Lu^C)_k=2$. For uniqueness, the difference of two points at which the chosen inequalities are tight vanishes on $C\cup\{0,d\}$ and is affine between consecutive points of this set, so it is zero.
\end{proof}

The uniqueness argument also shows that the normal vectors of the $d-1$ inequalities tight at each $u^C$ are linearly independent: their affine equations have the unique solution $u^C$. Lemma~\ref{lem:reverse-layered} and the vertex dictionary therefore identify the vertices of $\cQ_d$ with the generating points of $P_d$. Since $\cQ_d$ is a bounded intersection of half-spaces, it is the convex hull of its vertices, and the inverse deficit map sends $\cQ_d$ onto $P_d$. We record the inequalities tight at a vertex by its \emph{signature} in $\{A,B\}^{d-1}$, which has $B$ exactly at the cuts. A \emph{partial signature} in $\{A,B,*\}^{d-1}$ fixes the choice at the positions labelled $A$ or $B$ and leaves the positions labelled $*$ free. A partial signature $\sigma$ \emph{extends} $\sigma'$ when it agrees with $\sigma'$ at every position fixed by $\sigma'$; signatures are the partial signatures without free positions.

\begin{samepage}
\begin{theorem}\label{thm:cube-face-lattice}
Let $d\ge2$. The inequalities \eqref{eq:A-facet} and \eqref{eq:B-facet} for $1\le k<d$, together with the equation $\sum_ix_i=d(d+1)/2$, describe $P_d$ irredundantly. A partial signature fixing $r$ positions defines the face of $P_d$ on which the chosen inequalities hold with equality. This face has dimension $d-1-r$, and its vertices are the $2^{d-1-r}$ reverse-layered permutations whose signatures extend the partial signature. Every nonempty face arises from exactly one partial signature, and the face of $\sigma$ is contained in the face of $\sigma'$ exactly when $\sigma$ extends $\sigma'$. In particular, $P_d$ is combinatorially a $(d-1)$-cube.
\end{theorem}
\end{samepage}

\begin{proof}
The vertex dictionary shows that the displayed system defines $P_d$. We work in deficit coordinates. Let $I_A,I_B\subseteq[d-1]$ be disjoint, put $r=|I_A|+|I_B|$, and impose $(Lu)_k=2$ for $k\in I_A$ and $u_k=0$ for $k\in I_B$. The resulting set $F$ is an exposed face, because it is the zero set of a sum of nonnegative slacks, the quantities $2-(Lu)_k$ and $u_k$ associated with the chosen inequalities. Since exactly one inequality of each pair is tight at every vertex, the vertices in $F$ are the points $u^C$ whose signatures extend the partial one. Let $v$ be the barycenter of these $2^{d-1-r}$ vertices. At a fixed position the chosen inequality is tight at $v$ and its partner is strict. At a free position both choices occur among the averaged vertices, so both slacks are positive at $v$.

The normal vectors of the chosen inequalities are linearly independent, being a subset of the $d-1$ linearly independent normal vectors of the inequalities tight at any vertex of $F$. Their equality conditions therefore cut out an affine subspace of dimension $d-1-r$. Every inequality that is not chosen is strict at $v$, so a small neighborhood of $v$ in this subspace lies in $F$, and $\dim F=d-1-r$. For $r=1$ this shows that each of the $2(d-1)$ inequalities defines a facet, and different inequalities define different facets; hence none can be removed.

Conversely, let $F$ be a nonempty face and choose $v\in\relint(F)$, where $\relint(F)$ denotes the interior of $F$ in its affine hull. A facet contains $F$ exactly when its inequality is tight at $v$. By Lemma~\ref{lem:signatures}, at most one inequality of each pair is tight at $v$, so the facets containing $F$ form a partial signature. The face of this signature is the intersection of the same facets, hence equals $F$, and the set of tight inequalities determines the signature. Inclusion of faces corresponds to inclusion of vertex sets, that is, to extension of partial signatures. This is the face lattice of the cube $[0,1]^{d-1}$.
\end{proof}

\section{Two lattice models}\label{sec:models}

Two polytopes are \emph{lattice equivalent}, or unimodularly equivalent, if an
affine bijection between their affine hulls maps one onto the other and maps
the integer points of the first affine hull onto those of the second. The
facet description identifies $P_d$, up to lattice equivalence, with two
polytopes that were studied without reference to pattern avoidance. The first
is twice the part of a permutahedron that lies in one closed Weyl chamber, a
poset permutahedron of a chain in the sense of Black, Rehberg, and Sanyal.
After translation by minus the all-ones vector and before dilation by two, its lattice points are the tournament score sequences. Averaging over the
symmetric group, which acts on the full permutahedron, gives the Ehrhart series
in Section~\ref{sec:eulerian}. The second is the Newton polytope of the
discriminant of a univariate polynomial, whose face lattice and volume are
classical. This section proves both equivalences and derives the volume of
$P_d$ from the known volumes of the two models.

\subsection{The chain-poset permutahedron}

We use the convention of Black, Rehberg, and Sanyal: $\Pi_{\cC_d}$ is the intersection of the permutahedron of $(1,\ldots,d)$ with the chamber of increasing coordinates. Put $R_d=\Pi_{\cC_d}-\mathbf1$. For the chain $1<\cdots<d$, the filters in their defining inequalities are suffixes. Subtracting their upper bounds from the total-sum equation gives lower bounds on prefixes. Thus, for $m\ge0$, a point $a=(a_1,\ldots,a_d)$ belongs to $mR_d$ precisely when \cite[Corollary~3.4]{BlackRehbergSanyal2025}
\begin{align}
 a_1&\le\cdots\le a_d,\label{eq:R-order}\\
 \sum_{i=1}^d a_i&=m\binom d2,\label{eq:R-total}\\
 \sum_{i=1}^k a_i&\ge m\binom k2\qquad(1\le k<d).\label{eq:R-prefix}
\end{align}

\begin{theorem}\label{thm:lattice-map}
The affine map
\[
 \Psi_d(x)_i=d+i-1-x_i
\]
is a lattice equivalence from $P_d$ to $2R_d$.
\end{theorem}

\begin{proof}
The affine equation for $P_d$ gives $\sum_i\Psi_d(x)_i=d(d-1)=2\binom d2$. Also,
\[
 \Psi_d(x)_{k+1}-\Psi_d(x)_k=x_k-x_{k+1}+1,
\]
so \eqref{eq:A-facet} is equivalent to \eqref{eq:R-order}. Finally,
\[
 \sum_{i=1}^k\Psi_d(x)_i=k(d-1)+\binom{k+1}{2}-S_k(x),
\]
and \eqref{eq:B-facet} is equivalent to
\[
 \sum_{i=1}^k\Psi_d(x)_i\ge2\binom k2.
\]
The linear part is $-I$, an automorphism of the difference lattice $\mathsf A_{d-1}$.
\end{proof}

Equivalently, if $a\in2R_d$, the prefix map
\[
 u_k=\sum_{i=1}^k a_i-k(k-1)
\]
identifies the prefix inequalities with $u_k\ge0$ and the order inequalities with $(Lu)_k\le2$. Its inverse is
\[
 a_k=u_k-u_{k-1}+2(k-1).
\]

An affine bijection maps vertices to vertices, so Theorem~\ref{thm:lattice-map} also describes the vertices of $2R_d$. If positions $p+1,\ldots,p'$ form a block of a reverse-layered permutation $x$, these positions carry the values $d-p'+1,\ldots,d-p$ in increasing order, and therefore $\Psi_d(x)_i=p+p'-1$ for $p<i\le p'$. Every vertex of $R_d$ is thus constant on the blocks of a composition of $[d]$ into intervals, with value $(p+p'-1)/2$ on the block $\{p+1,\ldots,p'\}$. A block of length two gives a nonintegral value. Hence, for $d\ge2$, the smallest positive integer $q$ for which $qR_d$ has integral vertices is $q=2$; this number is the \emph{denominator} of $R_d$.

\subsection{The Newton polytope of the discriminant}

The discriminant $D_d(c_0,\ldots,c_d)$ of the polynomial $c_0+c_1t+\cdots+c_dt^d$ is homogeneous of degree $2(d-1)$ and quasihomogeneous of weight $d(d-1)$ when $c_k$ has weight $k$. Its Newton polytope $N_d\subset\R^{d+1}$, the convex hull of the exponent vectors of its monomials, therefore lies in the plane
\[
 \sum_{k=0}^d b_k=2(d-1),\qquad \sum_{k=0}^d kb_k=d(d-1).
\]
Gelfand, Kapranov, and Zelevinsky proved that $N_d$ is combinatorially a $(d-1)$-cube and described its vertices \cite{GKZ1990}, \cite[Chapter~12, Theorem~2.2]{GKZ1994}. A vertex corresponds to a subset $C\subseteq[d-1]$, viewed as the subdivision of the segment $[0,d]$ at the points of $C$. If $k\in C$ has the neighbors $k^-<k<k^+$ in $C\cup\{0,d\}$, the vertex has $b_k=k^+-k^-$; its endpoint coordinates $b_0$ and $b_d$ are the lengths of the first and last segments minus one, and all other coordinates vanish. Orevkov states this description in exactly this form \cite[Section~3]{Orevkov1999} and computes the volume of $N_d$ \cite[Theorem~1]{Orevkov1999}. Black, Rehberg, and Sanyal observe that $N_d$ is unimodularly equivalent to $2\Pi_{\cC_d}$ \cite[p.~2]{BlackRehbergSanyal2025}. The next proposition makes the equivalence explicit on the pattern-avoiding model.

\begin{proposition}\label{prop:discriminant}
For $d\ge2$, the affine map $x\mapsto b=(b_0,\ldots,b_d)$ given by
\[
 b_0=d-x_1,\qquad b_k=x_k-x_{k+1}+1\ (1\le k<d),\qquad b_d=x_d-1,
\]
is a lattice equivalence from $P_d$ onto $N_d$. It sends the reverse-layered permutation with cut set $C$ to the vertex of $N_d$ indexed by $C$.
\end{proposition}

\begin{proof}
On the plane $\sum_ix_i=d(d+1)/2$, the map satisfies $\sum_kb_k=2(d-1)$ and $\sum_kkb_k=d(d-1)$. Its inverse is the integral recursion $x_1=d-b_0$, $x_{k+1}=x_k+1-b_k$, so it identifies the integer points of the two affine planes. At the reverse-layered permutation with cut set $C$, consecutive positions inside a block have $x_{k+1}=x_k+1$, hence $b_k=0$ unless $k\in C$. If $k\in C$ has the neighbors $k^-<k<k^+$ in $C\cup\{0,d\}$, the block $\{k^-+1,\ldots,k\}$ ends with the value $x_k=d-k^-$ and the next block $\{k+1,\ldots,k^+\}$ starts with $x_{k+1}=d-k^++1$, as in the description of the vertices of $2R_d$ above; thus $b_k=k^+-k^-$. If the first block has length $\ell$, it starts with $x_1=d-\ell+1$, so $b_0=\ell-1$; symmetrically, $b_d$ is the length of the last block minus one. The vertices of $P_d$ are therefore mapped onto the vertices of $N_d$, and the convex hulls correspond.
\end{proof}

\begin{corollary}\label{cor:volume}
For every $d\ge1$, the normalized volume of $P_d$ is $2^{d-1}d^{d-3}$.
\end{corollary}

\begin{proof}
For $d=1$, $P_1$ is a lattice point of normalized volume one. Let $d\ge2$. Normalized volume is $(d-1)!$ times relative volume, the volume measured so that a fundamental parallelepiped of the affine lattice has volume one. Black, Rehberg, and Sanyal show that the poset permutahedron of a poset $P$ on $[d]$ with $e(P)$ linear extensions has relative volume $e(P)\,d^{d-2}/d!$ \cite[Corollary~4.2]{BlackRehbergSanyal2025}; for the chain, $e(\cC_d)=1$, and this is the relative volume $d^{d-2}$ of the permutahedron shared equally among the $d!$ Weyl chambers. Translation preserves relative volume and dilation by two multiplies it by $2^{d-1}$. Theorem~\ref{thm:lattice-map} therefore gives
\[
 \vol_{\mathrm{norm}}P_d=(d-1)!\,2^{d-1}\frac{d^{d-2}}{d!}=2^{d-1}d^{d-3}.
\]
Proposition~\ref{prop:discriminant} gives the same value through Orevkov's theorem $\vol N_d=2^{d-1}d^{d-2}/d!$, where volume is again relative volume \cite[Theorem~1]{Orevkov1999}.
\end{proof}

Together with Theorem~\ref{thm:cube-face-lattice}, the two lattice equivalences and Corollary~\ref{cor:volume} complete the proof of Theorem~\ref{thm:intro-geometry}.

\section{Interior lattice points and strong score sequences}\label{sec:tournaments}

The leading coefficient of an Ehrhart numerator counts lattice points in the
relative interior of a dilate. For $d\ge3$, this coefficient for $P_d$ is
$[t^{d-1}]H_d(t)$, and the chain model of Section~\ref{sec:models} turns the
corresponding interior points into ordered integer vectors satisfying strict
prefix inequalities. A shift of the coordinates converts these inequalities
into the strict form of Landau's criterion, which characterizes the score
sequences of strong tournaments. In this range, the same number is the leading coefficient of
the rational Ehrhart numerator of $\Pi_{\cC_d}$, the quantity in a conjecture of
Rehberg. We first recall the form of Ehrhart reciprocity that compares these
two numerators.

Let $Q\subset\R^N$ be a rational polytope of dimension $r$ and denominator
$q$, the least positive integer for which $qQ$ has integral vertices. Its Ehrhart series is a rational function
\[
 \sum_{m\ge0}\lvert mQ\cap\Z^N\rvert t^m
 =\frac{h^*_{\Z}(Q;t)}{(1-t^q)^{r+1}}
\]
with a polynomial numerator of degree less than $q(r+1)$ \cite[Theorem~3.23 and Lemma~3.24]{BeckRobins2015}. We
follow Rehberg's notation $h^*_{\Z}$ \cite[p.~1]{Rehberg2025}; for a
lattice polytope, $q=1$ and $h^*_{\Z}(Q;t)$ is the usual $h^*$-polynomial.
Ehrhart--Macdonald reciprocity for rational polytopes, in the form of \cite[Theorem~4.4]{BeckRobins2015}, gives
\[
 \sum_{m\ge1}\lvert\relint(mQ)\cap\Z^N\rvert t^m
 =\frac{t^{q(r+1)}h^*_{\Z}(Q;t^{-1})}{(1-t^q)^{r+1}}.
\]
Let $m_0$ be the smallest positive integer for which $\relint(m_0Q)$ contains a
lattice point. Comparing the lowest powers of $t$ on the two sides shows that
$h^*_{\Z}(Q;t)$ has degree $q(r+1)-m_0$ and that its leading coefficient is
$\lvert\relint(m_0Q)\cap\Z^N\rvert$. We apply this twice. For the lattice
polytope $P_d$, with $q=1$ and $r=d-1$, the comparison of the coefficients of
$t^1$ shows that $[t^{d-1}]H_d(t)$ is always the number of interior lattice
points of $P_d$; it is the leading coefficient exactly when $m_0=1$. For $d\ge2$,
the polytope $R_d$ has denominator $2$ by Section~\ref{sec:models}, and its
rational numerator is defined by
\[
 \sum_{m\ge0}\lvert mR_d\cap\Z^d\rvert t^m
 =\frac{h^*_{\Z}(R_d;t)}{(1-t^2)^d}.
\]
Since $R_d$ is an integral translate of $\Pi_{\cC_d}$, the two polytopes have
the same Ehrhart series and the same numerator.

\begin{proposition}\label{prop:strong-bijection}
Adopt the standard convention that the one-vertex tournament is strong. For every $d\ge1$, the maps
\[
 a_i=s_i+i-1,
 \qquad
 s_i=a_i-i+1
\]
are inverse bijections between strong tournament score sequences and $\relint(2R_d)\cap\Z^d$.
\end{proposition}

\begin{proof}
A score sequence is counted here once as a nondecreasing integer vector, irrespective of how many labeled tournaments realize it. Landau's inequalities characterize tournament score sequences \cite{Landau1953}, and their strict form characterizes the strongly connected case \cite[Theorem~9]{HararyMoser1966}. Both statements in the present sequence convention are given in \cite[Propositions~1 and~7]{Stockmeyer2023}. Thus a nondecreasing integer sequence $s$ is a strong tournament score sequence exactly when
\[
 \sum_{i=1}^d s_i=\binom d2,
 \qquad
 \sum_{i=1}^k s_i>\binom k2\quad(1\le k<d).
\]
The shift $a_i=s_i+i-1$ makes $a$ strictly increasing and changes the total from $\binom d2$ to $2\binom d2$. For every $k<d$ it gives
\[
 \sum_{i=1}^ka_i=\sum_{i=1}^ks_i+\binom k2>2\binom k2.
\]
These are precisely the strict versions of \eqref{eq:R-order} and \eqref{eq:R-prefix}, together with the unchanged total equation, so they characterize $\relint(2R_d)$. Conversely, a strictly increasing integral $a\in\relint(2R_d)$ yields a nondecreasing integral sequence $s_i=a_i-i+1$ satisfying the strict Landau inequalities. The two constructions are inverse.
\end{proof}

For $d=4$, $5$, and $6$, there are respectively $1$, $3$, and $7$ strong score sequences. These are the corresponding top coefficients in Table~\ref{tab:small-numerators}.

The first dilate $R_d$ has no interior lattice point for $d\ge2$: strict order and the first strict prefix inequality would force $a_i\ge i$, contradicting $\sum_ia_i=\binom d2$. For every $d\ge3$, orient the edges of a Hamiltonian cycle cyclically and the remaining edges arbitrarily; the resulting tournament contains a directed spanning cycle and is strong. Hence, for $d\ge3$, the smallest dilate of $R_d$ with an interior lattice point is the second one, and the first dilate of $P_d\cong_{\Z}2R_d$ already has one. Reciprocity therefore gives $\deg h^*_{\Z}(R_d;t)=2d-2$ and $\deg H_d(t)=d-1$, and both leading coefficients equal $\lvert\relint(2R_d)\cap\Z^d\rvert$. Proposition~\ref{prop:strong-bijection} identifies this number with the number of strong score sequences, which proves Theorem~\ref{thm:intro-tournament}.

Rehberg calls a score sequence irreducible when no proper prefix sum $\sum_{i\le k}s_i$ equals $\binom k2$; these are exactly the score sequences of strong tournaments \cite[Theorem~6.34]{Rehberg2025}, as in the strict criterion used in Proposition~\ref{prop:strong-bijection}. She conjectures that the number of irreducible score sequences on $d$ teams equals the leading coefficient of $h^*_{\Z}(\Pi_{\cC_d};t)$ \cite[Conjecture~6.36]{Rehberg2025}. The argument above proves this for $d\ge3$. For $d=1$, the polytope $R_1=\{0\}$ has denominator $1$ and numerator $1$, in agreement with the single score sequence $(0)$, which is strong by convention. For $d=2$, we have $\lvert mR_2\cap\Z^2\rvert=\lfloor m/2\rfloor+1$, so $h^*_{\Z}(R_2;t)=1+t$ has leading coefficient one, whereas no tournament on two vertices is strong. The coefficient of $t^{2d-2}$ vanishes in this case. Thus the identity equating $[t^{2d-2}]h^*_{\Z}(R_d;t)$ with the number of strong score sequences holds for every $d\ge1$. It agrees with the leading-coefficient formulation of the conjecture for $d=1$ and $d\ge3$.

\section{The refined Eulerian expansion}\label{sec:eulerian}

The next objective is to convert the lattice-point formula into a form in
which the zeros of $H_d(t)$ can be controlled. We expand $H_d(t)$ in the
refined Eulerian polynomials. They form a compatible family
\cite[Remark~1]{BrentiWelker2008}, so an expansion with nonnegative coordinates
would already imply real-rootedness, as for lattice zonotopes
\cite{BeckJochemkoMcCullough2019}; compatibility is also the case $c=0$ of
Lemma~\ref{lem:repair} below. They also
behave well under the insertion of a new smallest or largest letter, which
drives both the expansion and the proof of compatibility. The main work of this
section is to derive the coordinates of $H_d(t)$ in this basis directly from
the lattice-point model.

For $1\le i\le d$, let
\begin{equation}\label{eq:refined-eulerian}
 A_{d,i}(t)=\sum_{\substack{\pi\in S_d\\\pi_1=i}}t^{\operatorname{des}(\pi)}
 =\sum_{\substack{\pi\in S_d\\\pi_d=d+1-i}}t^{\operatorname{des}(\pi)},
\end{equation}
with $\operatorname{des}(\pi)$ as defined in the Introduction. The first sum is the refinement used in \cite[Section~4]{NevoPetersenTenner2011}. The reverse-complement involution $\pi\mapsto(d+1-\pi_d,\ldots,d+1-\pi_1)$ preserves descents and proves the second equality. In the notation of Athanasiadis \cite[Example~2.3]{Athanasiadis2025}, the same polynomial is $p_{d-1,i-1}(t)$. The last-letter form is convenient for insertion arguments. Define
\begin{align}
 \mathsf L_df&=(1+(d-1)t)f+t(1-t)f',\label{eq:L-operator}\\
 \mathsf R_df&=t\bigl(df+(1-t)f'\bigr).\label{eq:R-operator}
\end{align}
Two elementary insertions give
\begin{equation}\label{eq:insertion-identities}
 A_{d+1,i}(t)=\mathsf L_dA_{d,i}(t),
 \qquad
 A_{d+1,i+1}(t)=\mathsf R_dA_{d,i}(t)
 \qquad(1\le i\le d).
\end{equation}

For the first identity, increase every letter of a permutation counted by $A_{d,i}(t)$ by one and insert a new minimum in one of the first $d$ slots. A permutation with $k$ descents has $k+1$ slots that preserve the number of descents and $d-1-k$ slots that add one, producing
$((k+1)t^k+(d-1-k)t^{k+1})$. Summing over $k$ is exactly the operator $\mathsf L_d$. For the second identity, insert a new maximum in one of the first $d$ slots. There are $k$ slots that preserve $k$ descents and $d-k$ slots that produce $k+1$, which gives $kt^k+(d-k)t^{k+1}$ and hence $\mathsf R_d$. The last letter has the value required by \eqref{eq:refined-eulerian} in both constructions.

Write $L_{P_d}(m)=\lvert mP_d\cap\Z^d\rvert$ for the Ehrhart polynomial of
$P_d$. To compute it, we count lattice points of the chain model, whose
coordinates are nondecreasing, $a_1\le\cdots\le a_d$. This region is a closed
chamber for the action of $S_d$ by coordinate permutations, and every orbit of
the full permutahedron has a unique representative in it, so Burnside
averaging replaces the chamber count by fixed-point counts. Each fixed-point
count is a sum over set partitions of the cycles of a permutation of a
product of factors attached to the blocks, and this multiplicative structure
makes the exponential formula applicable. We now carry out this reduction
explicitly. If $\lambda=(\ell_1,\ldots,\ell_r)\vdash d$, let $m_j(\lambda)$ be the number of parts equal to $j$ and put
\[
 z_\lambda=\prod_{j\ge1}j^{m_j(\lambda)}m_j(\lambda)!
\]
and let $\operatorname{Part}([r])$ denote the set of set partitions of the $r$ cycles. Let $\Pi_d^{\mathrm{full}}$ be the permutahedron of $(0,1,\ldots,d-1)$. The polytope $R_d$ is the intersection of $\Pi_d^{\mathrm{full}}$ with this chamber, so every $S_d$-orbit of lattice points of $2m\Pi_d^{\mathrm{full}}$ has a unique nondecreasing representative in $2mR_d$. The lattice equivalence $P_d\cong_{\mathbb Z}2R_d$ and Burnside's orbit-counting lemma \cite[Chapter~7]{Stanley2018Algebraic} therefore give
\[
 L_{P_d}(m)=\frac1{d!}\sum_{\sigma\in S_d}
 \left|(2m\Pi_d^{\mathrm{full}})^\sigma\cap\mathbb Z^d\right|.
\]
The permutahedron in Ardila, Supina, and Vindas-Mel\'endez \cite{ArdilaSupinaVindas2020} is a translate of $\Pi_d^{\mathrm{full}}$ by $\mathbf1$. Translation commutes with the coordinate action, and each integral dilate translates the fixed-point sets by an integer vector. Their Theorem~1.1 therefore applies to our convention. For a permutation with cycle lengths $\lambda=(\ell_1,\ldots,\ell_r)$, the even-dilation branch at $2m$ sums over all set partitions of the cycles. Grouping the Burnside average by cycle type gives
\begin{equation}\label{eq:burnside-cycle-type}
 L_{P_d}(m)=\sum_{\lambda\vdash d}\frac1{z_\lambda}
 \sum_{\pi\in\operatorname{Part}([r])}
 \prod_{B\in\pi}
 \gcd(\ell_i:i\in B)
 \left(\sum_{i\in B}\ell_i\right)^{|B|-2}
 (2m)^{r-|\pi|}.
\end{equation}
For a singleton block the factor before $(2m)^0$ is one. This convention follows directly from the displayed expression.

Retaining only the partitions with a single block in \eqref{eq:burnside-cycle-type} and replacing $2m$ by a variable $y$ defines the \emph{connected series}
\begin{equation}\label{eq:K-series}
 K(u,y)=\sum_{d\ge1}u^d\sum_{\lambda=(\ell_1,\ldots,\ell_r)\vdash d}
 \frac{\gcd(\ell_1,\ldots,\ell_r)\,d^{\,r-2}}{z_\lambda}\,y^{r-1}.
\end{equation}
As in the exponential formula, $\exp K$ reassembles the blocks. Define $C_{d,q}$ by
\begin{equation}\label{eq:K-def}
 \exp K(u,y)=\sum_{d,q\ge0}C_{d,q}u^dy^q,
\end{equation}
where the constant term is $C_{0,0}=1$. Write $\varphi$ for Euler's totient function. The next lemma derives the closed form
\begin{equation}\label{eq:connected-burnside}
 [u^d]K(u,y)=\frac1{d^2}\sum_{h\mid d}\varphi(h)
 \binom{(1+y)d/h-1}{d/h-1}.
\end{equation}

All coefficients in \eqref{eq:K-def} are ordinary coefficients in $u$ and $y$. The cycle-index factors $1/z_\lambda$ in \eqref{eq:burnside-cycle-type} already contain the label normalization.

\begin{lemma}\label{lem:connected-cycle-index}
The coefficients of the series \eqref{eq:K-series} are given by \eqref{eq:connected-burnside}. Moreover,
\begin{equation}\label{eq:ehrhart-C}
 L_{P_d}(m)=\sum_q2^qC_{d,q}m^q.
\end{equation}
\end{lemma}

\begin{proof}
The contribution of one block of $r$ cycles with lengths $\ell_1,\ldots,\ell_r$ and total length $d$ is
\[
 \gcd(\ell_1,\ldots,\ell_r)d^{r-2}y^{r-1}.
\]
The cycle-index weight of these cycles is $(r!\ell_1\cdots\ell_r)^{-1}$ after summing over ordered length lists. Hence the connected series has coefficient
\begin{equation}\label{eq:connected-composition}
 [u^d]K(u,y)=
 \sum_{r\ge1}\frac{d^{r-2}y^{r-1}}{r!}
 \sum_{\substack{\ell_1+\cdots+\ell_r=d\\ \ell_i\ge1}}
 \frac{\gcd(\ell_1,\ldots,\ell_r)}{\ell_1\cdots\ell_r}.
\end{equation}
Use $\gcd(\ell_1,\ldots,\ell_r)=\sum_{h\mid\ell_1,\ldots,\ell_r}\varphi(h)$. For a fixed divisor $h\mid d$, put $n=d/h$ and $\ell_i=h\ell_i'$. Since
\[
 \sum_{\substack{\ell_1'+\cdots+\ell_r'=n\\\ell_i'\ge1}}
 \frac1{\ell_1'\cdots\ell_r'}
 =[x^n]\bigl(-\log(1-x)\bigr)^r,
\]
equation \eqref{eq:connected-composition} becomes
\begin{align*}
 [u^d]K(u,y)
 &=\sum_{h\mid d}\frac{\varphi(h)}{h^2n^2}
 [x^n]\frac{(1-x)^{-ny}-1}{y}\\
 &=\frac1{d^2}\sum_{h\mid d}\varphi(h)
 \binom{(1+y)d/h-1}{d/h-1}.
\end{align*}
The last equality uses
\[
 \frac1y[x^n](1-x)^{-ny}
 =\frac1y\binom{ny+n-1}{n}
 =\binom{ny+n-1}{n-1},
\]
interpreted polynomially at $y=0$. This proves \eqref{eq:connected-burnside}.

A set partition of the cycles is an unordered collection of nonempty blocks. Expanding $\exp K(u,y)$ supplies the factor $1/k!$ for a collection of $k$ blocks; within each block, $1/(r!\ell_1\cdots\ell_r)$ accounts for its cycles. These products give exactly the weights in \eqref{eq:burnside-cycle-type}. This is the labeled-set exponential formula \cite[Theorem~II.1]{FlajoletSedgewick2009}. Since $r-|\pi|=\sum_{B\in\pi}(|B|-1)$, setting $y=2m$ gives
\[
 L_{P_d}(m)=[u^d]\exp K(u,2m)
 =\sum_{q\ge0}C_{d,q}(2m)^q,
\]
which is \eqref{eq:ehrhart-C}.
\end{proof}

The Ehrhart polynomial is now expressed as a sum of powers $m^q$. The
classical Worpitzky identity \cite[Section~1.4]{Stanley2012} converts
these powers into an Eulerian numerator. The refined identity below appears in Brenti and Welker \cite[equation~(4)]{BrentiWelker2008}, with their parameters $d,j$ equal to our $d-1,i-1$. We give an insertion proof and then use the identity to express the numerator in the basis $A_{d,i}(t)$.

Define the refined Eulerian coordinates
\begin{equation}\label{eq:magic-coordinates}
 a_{d,j}=\sum_{q=0}^j(-1)^{j-q}
 \binom{d-1-q}{j-q}2^qC_{d,q}.
\end{equation}
For completeness, define $\mathcal E_q(t)$ by
\[
 \sum_{m\ge0}m^qt^m=\frac{\mathcal E_q(t)}{(1-t)^{q+1}},
 \qquad \mathcal E_0(t)=1.
\]
Thus $\mathcal E_q(t)=t\sum_{\pi\in S_q}t^{\operatorname{des}(\pi)}$ for $q\ge1$; the factor $t$ records the convention used in this power-sum identity.
\begin{lemma}\label{lem:refined-worpitzky}
For $1\le i\le d$,
\begin{equation}\label{eq:refined-series}
 \frac{A_{d,i}(t)}{(1-t)^d}
 =\sum_{m\ge0}m^{i-1}(m+1)^{d-i}t^m,
\end{equation}
where $0^0=1$. Consequently, for $0\le q\le d-1$,
\begin{equation}\label{eq:refined-worpitzky}
 \mathcal E_q(t)(1-t)^{d-1-q}
 =\sum_{j=q}^{d-1}(-1)^{j-q}\binom{d-1-q}{j-q}A_{d,j+1}(t).
\end{equation}
\end{lemma}

\begin{proof}
Equation \eqref{eq:refined-series} holds for $d=1$. Suppose its right-hand side is $F(t)$. A direct differentiation of $A_{d,i}(t)=(1-t)^dF(t)$ gives
\[
 \frac{\mathsf L_dA_{d,i}(t)}{(1-t)^{d+1}}=F+tF',
 \qquad
 \frac{\mathsf R_dA_{d,i}(t)}{(1-t)^{d+1}}=tF'.
\]
Write $f_m=m^{i-1}(m+1)^{d-i}$ for the coefficient of $t^m$ in
$F(t)$. Then
$[t^m](F+tF')=(m+1)f_m=m^{i-1}(m+1)^{d+1-i}$, whereas
$[t^m]tF'=mf_m=m^i(m+1)^{d-i}$.
By \eqref{eq:insertion-identities}, these are respectively the series for
$A_{d+1,i}(t)$ and $A_{d+1,i+1}(t)$. As $i$ ranges from $1$ to $d$, the two
identities cover every index from $1$ to $d+1$, completing the induction.

Divide the right-hand side of \eqref{eq:refined-worpitzky} by $(1-t)^d$ and use \eqref{eq:refined-series}. The coefficient of $t^m$ becomes
\begin{align*}
 &\sum_{j=q}^{d-1}(-1)^{j-q}\binom{d-1-q}{j-q}
 m^j(m+1)^{d-1-j}\\
 &\qquad=m^q\sum_{k=0}^{d-1-q}(-1)^k\binom{d-1-q}{k}
 m^k(m+1)^{d-1-q-k}=m^q.
\end{align*}
Thus the quotient is $\sum_{m\ge0}m^qt^m=\mathcal E_q(t)/(1-t)^{q+1}$, proving the result.
\end{proof}

Equation~\eqref{eq:refined-series} also shows that $A_{d,1},\ldots,A_{d,d}$ form a basis of $\mathbb R_{\le d-1}[t]$. Indeed, the polynomials $m^j(m+1)^{d-1-j}$ are linearly independent: divide a putative relation by $(m+1)^{d-1}$ and set $z=m/(m+1)$. Thus the refined Eulerian coordinates are unique. Substituting \eqref{eq:ehrhart-C} into the Ehrhart series gives
\[
 \frac{H_d(t)}{(1-t)^d}
 =\sum_{q=0}^{d-1}2^qC_{d,q}\frac{\mathcal E_q(t)}{(1-t)^{q+1}}.
\]
After multiplying by $(1-t)^d$, Lemma~\ref{lem:refined-worpitzky} identifies the coefficient of each $A_{d,j+1}(t)$ with \eqref{eq:magic-coordinates}. Therefore
\begin{equation}\label{eq:H-eulerian-expansion}
 H_d(t)=\sum_{j=0}^{d-1}a_{d,j}A_{d,j+1}(t).
\end{equation}

The coefficients $a_{d,j}$ need not be positive. Write $\eta_n=\sum_{k=1}^n k^{-1}$ and $\eta_0=0$. From \eqref{eq:connected-burnside},
\[
 [u^n]\partial_yK(u,0)=\frac1n\sum_{h\mid n}\frac{\varphi(h)}h \eta_{n/h-1}.
\]
Since $\exp K(u,0)=(1-u)^{-1}$, differentiation of \eqref{eq:K-def} gives
\[
 C_{7,1}=\sum_{n=1}^7\frac1n\sum_{h\mid n}\frac{\varphi(h)}h \eta_{n/h-1}=\frac{89}{30}.
\]
Together with $C_{7,0}=1$, formula \eqref{eq:magic-coordinates} yields
\[
 a_{7,1}=-6+2C_{7,1}=-\frac1{15},
 \qquad 7!a_{7,1}=-336.
\]
Thus coefficientwise positivity in \eqref{eq:H-eulerian-expansion} cannot establish real-rootedness throughout the family. The exact data of Section~\ref{sec:finite} show that $a_{d,1}<0$ for every $7\le d\le1000$ and that $a_{d,1}$ is the only negative coordinate in this range. This motivates an enlargement of the cone generated by the refined Eulerian basis in which a negative coordinate can be compensated by the following ones.

\section{A compatible cone of adjacent differences}\label{sec:repair}

The refined Eulerian polynomials form a compatible family, so a nonnegative
expansion in this basis implies real-rootedness. The negative coordinate
$a_{7,1}$ places $H_7(t)$ outside the cone they generate, and we therefore
look for a larger cone whose elements are still real-rooted. We enlarge the
family by the differences $A_{d,r+1}-cA_{d,r}$ of adjacent basis elements, in
which the parameter $c\ge0$ measures how much of one basis element may be
subtracted from the next. The first task is to determine exactly how large $c$
may be while compatibility holds for all $d\ge4$ simultaneously. The second is to express the
membership of $H_d(t)$ in the enlarged cone through its coordinates.

We regard the zero polynomial as real-rooted. A finite family of real
polynomials is \emph{compatible} if every nonnegative linear combination
is real-rooted. For nonzero real-rooted polynomials $f$ and $g$ with
positive leading coefficients, write $f\preceq g$ if
$\deg f\le\deg g\le\deg f+1$ and their zeros, listed in decreasing order
as $\theta_1\ge\theta_2\ge\cdots$ for $f$ and
$\tau_1\ge\tau_2\ge\cdots$ for $g$, satisfy
\[
 \cdots\le\theta_2\le\tau_2\le\theta_1\le\tau_1.
\]
Thus the rightmost zero belongs to $g$; equalities allow common or repeated
zeros. If $f$ is a positive constant, then $f\preceq g$ for every positive
constant $g$ and for every real-rooted $g$ of degree one. We also set $0\preceq g$ and $g\preceq0$ for
every real-rooted $g$. This is the interlacing convention used in
\cite[Section~2]{Athanasiadis2025}; see also
\cite[Section~2.4]{SavageVisontai2015}.

The following elementary criterion expresses interlacing through partial
fractions. We include its proof because we use the weak form, which allows
common zeros and the zero polynomial. Its simple-root hypothesis concerns only
the polynomial $g$.

\begin{lemma}\label{lem:residue-criterion}
Let $g$ be a polynomial of degree $N\ge1$ with positive leading coefficient and $N$ simple real zeros,
and let $f\in\mathbb R[t]$ have degree at most $N$. Then $f=0$, or $f$
has positive leading coefficient and $f\preceq g$, if and only if
\begin{equation}\label{eq:residue-order}
 \frac{f(t)}{g(t)}=b+\sum_{\gamma:g(\gamma)=0}
 \frac{\lambda_\gamma}{t-\gamma},
 \qquad b\ge0,\quad\lambda_\gamma\ge0.
\end{equation}
Consequently, polynomials satisfying this condition for the same $g$
form a compatible cone.
\end{lemma}

\begin{proof}
Partial fractions give
$\lambda_\gamma=f(\gamma)/g'(\gamma)$. If $f\preceq g$, the alternating
order of the zeros shows that these ratios are nonnegative, and the
constant term at infinity is nonnegative because the leading
coefficients are positive.

Conversely, first suppose that $b$ and all the residues are positive.
The quotient in \eqref{eq:residue-order} is strictly decreasing between
successive poles. It has one zero between each pair of consecutive
zeros of $g$, one zero to the left of the smallest zero of $g$, and
no zero to the right of the largest one. Hence its numerator interlaces
$g$. Allowing $b$ or individual residues to tend to zero gives the weak
inequalities, including common zeros. More explicitly, if $b=0$ and
some residue is positive, the numerator has degree $N-1$, with one zero
in each of the $N-1$ closed intervals between successive zeros of $g$;
if all residues vanish, the numerator is $bg$. These cases also include
the zero polynomial. Finally, nonnegative linear combinations preserve
all the inequalities in \eqref{eq:residue-order}, which proves the last
assertion.
\end{proof}

For $1\le r\le d-1$ and $c\ge0$, define the adjacent difference
\[
 E_{d,c}^{(r)}(t)=A_{d,r+1}(t)-cA_{d,r}(t),
\]
and put $\cstar=2-\sqrt3$ throughout. The cone generated by the family in the next lemma contains polynomials with negative refined Eulerian coordinates, and all its elements are real-rooted. The lemma determines the largest parameter for which this holds for all $d\ge4$ simultaneously.

\begin{lemma}\label{lem:repair}
For $d\ge4$ and $0\le c\le\cstar$, the family
\begin{equation}\label{eq:compatible-family}
 \Gamma_{d,c}=\{A_{d,1}(t),\ldots,A_{d,d}(t)\}
 \cup\{E_{d,c}^{(r)}(t):2\le r\le d-2\}
\end{equation}
is compatible. The number $\cstar=2-\sqrt3$ is the largest $c\ge0$ for which $\Gamma_{d,c}$ is compatible for all $d\ge4$ simultaneously.
\end{lemma}

\begin{proof}
We prove the stronger assertion that $f\preceq g_d$ for every $f\in\Gamma_{d,c}$, where
\[
 g_d(t)=A_{d,d}(t)=tA_{d,1}(t).
\]
Deleting the initial letter shows that $A_{d,1}(t)$ is the Eulerian polynomial of $S_{d-1}$, and deleting a final $1$, which always follows a descent, gives $A_{d,d}(t)=tA_{d,1}(t)$. Lemma~\ref{lem:residue-criterion} then gives compatibility, once we know that $g_d$ has simple real zeros. We therefore prove by induction on $d\ge4$ that $g_d$ has $d-1$ simple zeros, all nonpositive, and that $f\preceq g_d$ for every $f\in\Gamma_{d,c}$. The first assertion is the classical real-rootedness of Eulerian polynomials, which is also the case of a simplex in \cite[Theorem~2]{BrentiWelker2008}; we reprove it because the function $Q$ below serves both purposes. The residue characterization of Lemma~\ref{lem:residue-criterion} carries the second assertion through the insertion operators.

We first verify $d=4$. Put $\alpha=-2-\sqrt3$ and $\beta=-2+\sqrt3$. Then
\begin{align*}
 A_{4,1}(t)&=t^2+4t+1, & A_{4,2}(t)&=2t(t+2),\\
 A_{4,3}(t)&=2t(2t+1), & A_{4,4}(t)&=t(t-\alpha)(t-\beta),
\end{align*}
and
\[
 E_{4,c}^{(2)}(t)=2t\bigl((2-c)t+(1-2c)\bigr).
\]
For $0\le c\le\cstar$, its nonzero zero
\[
 \xi_c=-\frac{1-2c}{2-c}
\]
lies in $[-1/2,\beta]$. The nonzero zeros of $A_{4,2}(t)$ and $A_{4,3}(t)$ are $-2$ and $-1/2$, respectively. The zeros $\alpha<\beta<0$ of $g_4$ are simple, and consequently $f\preceq g_4=A_{4,4}(t)$ for every $f\in\Gamma_{4,c}$.

Assume now that $g=g_d$ has $d-1$ simple zeros $\gamma\le0$ and that $f\preceq g$ for every $f\in\Gamma_{d,c}$. Set
\[
 G=\mathsf R_dg=A_{d+1,d+1}(t),\qquad q=\frac{g'}g=\sum_\gamma\frac1{t-\gamma},\qquad Q=d+(1-t)q,
\]
so that $G=tgQ=t\bigl(dg+(1-t)g'\bigr)$. The function $Q$ is continuous on each of the $d-1$ open intervals bounded by consecutive zeros of $g$, or by $-\infty$ and the smallest zero, and $1-t>0$ there. Consequently $Q$ tends to $-\infty$ at the right end of each interval, to $+\infty$ at the left end of each bounded interval, and to $d-(d-1)=1$ as $t\to-\infty$, so it has a zero in each interval. At a negative zero $\gamma$ of $g$ we have $G(\gamma)=\gamma(1-\gamma)g'(\gamma)\ne0$, and $G(t)=t\bigl(dg(t)+(1-t)g'(t)\bigr)$ vanishes simply at $t=0$ because $g'(0)\ne0$. Since $G$ has degree $d$, its zeros are $0$ and exactly one zero of $Q$ in each interval, and all of them are simple. This completes the induction for the first assertion. The negative zeros of $G$ are precisely the zeros of $Q$. If $\rho<0$ is such a zero, then
\[
 q(\rho)=-\frac d{1-\rho}.
\]
The Cauchy--Schwarz inequality over the $d-1$ zeros of $g$ gives
\[
 \sum_\gamma\frac1{(\rho-\gamma)^2}
 \ge \frac{d^2}{(d-1)(1-\rho)^2},
\]
and hence
\[
 Q'(\rho)=-q(\rho)-(1-\rho)\sum_\gamma\frac1{(\rho-\gamma)^2}
 \le-\frac d{(d-1)(1-\rho)}<0.
\]
For $f\in\Gamma_{d,c}$, write $h=f/g$. Direct calculation gives
\begin{align*}
 \mathsf L_df&=hG+(1-t)g(h+th'),\\
 \mathsf R_df&=hG+t(1-t)gh'.
\end{align*}
The residue expansion of $h$ implies
\[
 h'=-\sum_\gamma\frac{\lambda_\gamma}{(t-\gamma)^2}\le0,
 \qquad
 h+th'=b+\sum_\gamma\frac{-\gamma\lambda_\gamma}{(t-\gamma)^2}\ge0.
\]
At a negative zero $\rho$ of $G$, the two residues are
\[
 \operatorname*{Res}_{t=\rho}\frac{\mathsf L_df}{G}
 =\frac{(1-\rho)(h(\rho)+\rho h'(\rho))}{\rho Q'(\rho)}\ge0,
 \qquad
 \operatorname*{Res}_{t=\rho}\frac{\mathsf R_df}{G}
 =\frac{(1-\rho)h'(\rho)}{Q'(\rho)}\ge0.
\]
At the root $t=0$ of $G$, the first residue is $f(0)/g'(0)=\lambda_0\ge0$, while the second quotient is regular because $\mathsf R_df$ contains a factor $t$. If $b$ is the constant in \eqref{eq:residue-order}, direct expansion at infinity gives constant terms $0$ and $b$, respectively, for $(\mathsf L_df)/G$ and $(\mathsf R_df)/G$. All residues and constants in the two partial-fraction expansions are therefore nonnegative, and Lemma~\ref{lem:residue-criterion} gives $\mathsf L_df\preceq G$ and $\mathsf R_df\preceq G$. This step does not use the value of $c$, which enters only through the base case.

The insertion identities and linearity give
\begin{align*}
 A_{d+1,i}(t)&=\mathsf L_dA_{d,i}(t) &&(1\le i\le d),\\
 A_{d+1,i}(t)&=\mathsf R_dA_{d,i-1}(t) &&(2\le i\le d+1),\\
 E_{d+1,c}^{(r)}&=\mathsf L_dE_{d,c}^{(r)} &&(2\le r\le d-2),\\
 E_{d+1,c}^{(r)}&=\mathsf R_dE_{d,c}^{(r-1)} &&(3\le r\le d-1).
\end{align*}
These images cover $\Gamma_{d+1,c}$, so induction proves compatibility for $0\le c\le\cstar$.

It remains to prove sharpness for every $c>\cstar$. Choose
\[
 \cstar<\delta<\min\{c,1/2\}.
\]
The roots of $E_{4,\delta}^{(2)}$ are $\xi_\delta$ and $0$, with
\[
 \beta<\xi_\delta<0,
\]
whereas the roots of $A_{4,1}(t)$ are $\alpha<\beta$. A concrete obstruction is the positive combination
\[
 \delta A_{4,1}(t)+E_{4,\delta}^{(2)}(t)
 =(4-\delta)t^2+2t+\delta.
\]
Its discriminant is $4(\delta^2-4\delta+1)<0$, because
$\cstar<\delta<1/2$ and the zeros of $\delta^2-4\delta+1$ are $2\pm\sqrt3$. Thus it has a pair of nonreal conjugate zeros. Since
\[
 E_{4,\delta}^{(2)}=E_{4,c}^{(2)}+(c-\delta)A_{4,2}(t),
\]
the same polynomial is a nonnegative linear combination of members of $\Gamma_{4,c}$. Thus $\Gamma_{4,c}$ is incompatible for every $c>\cstar$, which proves the asserted uniform optimality.
\end{proof}

The base-case transition is isolated in Figure~\ref{fig:repair-threshold}. In Fig.~\ref{fig:repair-threshold}(a), the zeros alternate. In Fig.~\ref{fig:repair-threshold}(b), two zeros coincide at $c=\cstar$; in Fig.~\ref{fig:repair-threshold}(c), their order prevents compatibility.

\begin{figure}[H]
\centering
\includegraphics[width=.94\textwidth]{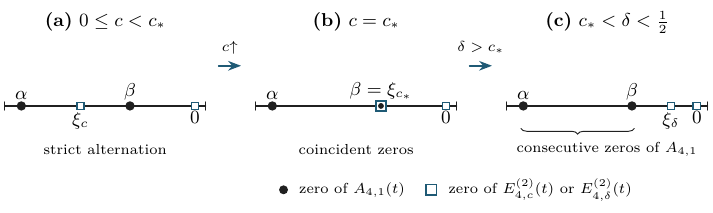}
\caption{The root collision in the family $\Gamma_{4,c}$ of Lemma~\ref{lem:repair}. Filled circles are the fixed zeros $\alpha$ and $\beta$ of $A_{4,1}(t)$; open squares mark $\xi_c$ (or $\xi_\delta$) and $0$, the zeros of $E^{(2)}_{4,c}(t)$ (or $E^{(2)}_{4,\delta}(t)$). The moving zero meets $\beta$ at $c_*=2-\sqrt3$, after which the two zeros of $A_{4,1}$ become consecutive. Horizontal spacing records order rather than metric distance.}
\label{fig:repair-threshold}
\end{figure}

At $c=\cstar$, the final adjacent difference $E_{d,\cstar}^{(d-1)}$ cannot be adjoined uniformly to the family in Lemma~\ref{lem:repair}. At $d=4$, the positive combination
\[
 A_{4,1}(t)+E_{4,\cstar}^{(3)}(t)
 =t^3+(4\sqrt3-3)t^2+(2\sqrt3+1)t+1
\]
has discriminant
\[
 4(623-360\sqrt3)<0.
\]
Thus the cone used below omits $E_{d,\cstar}^{(d-1)}$, which is why the tails defined next stop at $d-2$.

The value $\cstar$ is forced by $d=4$ alone. The induction step in the proof of Lemma~\ref{lem:repair} does not use the value of $c$, and an exact computation of the residues at $d=5$, where $g_5(t)=t(t+1)(t^2+10t+1)$, shows that $f\preceq g_5$ for every $f\in\Gamma_{5,c}$ whenever $0\le c\le(3-\sqrt6)/2\approx0.2753$; writing $p=(3-\sqrt6)/2$, the binding residue is $p(p-c)$, attained by $E^{(3)}_{5,c}/g_5$ at $-5+2\sqrt6$. Starting the induction at $d=5$, the family $\Gamma_{d,c}$ is therefore compatible for every $d\ge5$ and $0\le c\le(3-\sqrt6)/2$. We use the uniform value $\cstar$ throughout; it suffices for the finite certificate and for the analytic estimates.

Define the weighted tails
\begin{equation}\label{eq:tails}
 S_{d,r}(\cstar)=\sum_{j=r}^{d-2}\cstar^{j-r}a_{d,j},
 \qquad 1\le r\le d-2.
\end{equation}
They are the coordinates of $H_d(t)$ in the basis $A_{d,1},A_{d,2},E^{(2)}_{d,\cstar},\ldots,E^{(d-2)}_{d,\cstar},A_{d,d}$ of polynomials of degree at most $d-1$, which consists of members of $\Gamma_{d,\cstar}$, as the proof of the next corollary shows: the recursion $a_{d,r}=S_{d,r}(\cstar)-\cstar S_{d,r+1}(\cstar)$ absorbs each interior coordinate into an adjacent difference, and only the endpoint coordinates $a_{d,0}$ and $a_{d,d-1}$ remain outside the tails. These two coordinates are positive. Setting $y=0$ in \eqref{eq:connected-burnside} gives
\[
 [u^d]K(u,0)=\frac1{d^2}\sum_{h\mid d}\varphi(h)=\frac1d,
\]
and therefore $\exp K(u,0)=(1-u)^{-1}$. Hence $C_{d,0}=a_{d,0}=1$. At the other endpoint, $A_{d,d}(t)$ is the unique basis element in \eqref{eq:H-eulerian-expansion} of degree $d-1$, with leading coefficient one. Thus $a_{d,d-1}=[t^{d-1}]H_d(t)$, which is positive for $d\ge3$ by Theorem~\ref{thm:intro-tournament}.

\begin{corollary}\label{prop:cone}
For $d\ge4$, suppose that $a_{d,0}\ge0$, $a_{d,d-1}\ge0$, and every quantity in \eqref{eq:tails} is nonnegative. Then $H_d(t)$ has only real negative zeros.
\end{corollary}

\begin{proof}
With the convention $S_{d,d-1}(\cstar)=0$, the tail recursion is
\[
 a_{d,r}=S_{d,r}(\cstar)-\cstar S_{d,r+1}(\cstar).
\]
Substituting this identity for $1\le r\le d-2$ into
\eqref{eq:H-eulerian-expansion} and collecting adjacent basis elements
gives the exact decomposition
\begin{align}\label{eq:cone-decomposition}
 H_d(t)={}&a_{d,0}A_{d,1}(t)+S_{d,1}(\cstar)A_{d,2}(t)\\
 &+\sum_{r=2}^{d-2}S_{d,r}(\cstar)E_{d,\cstar}^{(r)}(t)
 +a_{d,d-1}A_{d,d}(t).\notag
\end{align}
Every scalar coefficient on the right is nonnegative, so
Lemma~\ref{lem:repair} gives real-rootedness. The coefficients of the
Ehrhart $h^*$-polynomial of a lattice polytope are nonnegative
\cite[Theorem~3.12]{BeckRobins2015}. Since $H_d(0)=1$, the polynomial has
no zero on the nonnegative real axis, and all of its zeros are therefore
negative.
\end{proof}

Thus real-rootedness has been reduced to the arithmetic inequalities
$S_{d,r}(\cstar)\ge0$. The remaining sections verify the same inequalities in
two complementary ways: exact reconstruction for the values $1\le d\le1000$
and uniform analytic estimates for all sufficiently large dimensions.

\section{An exact finite certificate}\label{sec:finite}

Corollary~\ref{prop:cone} reduces the finite-dimensional part of the
real-rootedness argument to a concrete verification problem: determine the
refined coordinates and check the sign of every weighted tail. Our objective
here is to make that step independently reproducible with exact integer
arithmetic. The CSV stores the coordinates, and the displayed recurrence and
pseudocode independently reconstruct and verify every entry for
$1\le d\le1000$.

The finite certificate is an exact table of the scaled refined coordinates
\[
 \alpha_{d,j}=d!a_{d,j}
 \qquad(1\le d\le1000,\ 0\le j\le d-1).
\]
It contains $\sum_{d=1}^{1000}d=500500$ integer entries. The file
\begin{center}
\nolinkurl{refined_coordinates_d1_d1000.csv}
\end{center}
is a CSV with the literal header \texttt{d,j,alpha\_scaled}.
Its third field stores $\alpha_{d,j}=d!a_{d,j}$, and the records are ordered
lexicographically by $(d,j)$. The first three coordinate vectors encode the identities
\[
 H_1(t)=1,
 \qquad H_2(t)=1,
 \qquad H_3(t)=(1+t)^2.
\]
The corresponding vectors are $(1)$, $(2,0)$, and $(6,0,6)$. For an independent reconstruction of every table entry, put
\[
 B_d(y)=d![u^d]K(u,y),
 \qquad
 \Delta_d(y)=d![u^d]\exp K(u,y),
 \qquad \Delta_0(y)=1.
\]
The definition \eqref{eq:K-series} shows that $B_d(y)$ has integer
coefficients. Since $d!/z_\lambda$ is the number of permutations of $[d]$ with cycle type $\lambda$, the polynomial $B_d(y)$ is a sum over these permutations of the weights
$\gcd(\ell_1,\ldots,\ell_r)d^{r-2}y^{r-1}$, and each weight has an integer coefficient; for $r=1$ it equals $1$. Differentiating the exponential gives the exact recurrence
\[
 \Delta_d(y)=\sum_{k=1}^d\binom{d-1}{k-1}B_k(y)\Delta_{d-k}(y).
\]
Starting from $\Delta_0(y)=1$, the recurrence proves inductively that every
$\Delta_d(y)$ also has integer coefficients.
Writing $\Delta_{d,q}=[y^q]\Delta_d(y)$, formula \eqref{eq:magic-coordinates} becomes
\begin{equation}\label{eq:scaled-magic-coordinates}
 \alpha_{d,j}=\sum_{q=0}^j(-1)^{j-q}
 \binom{d-1-q}{j-q}2^q\Delta_{d,q}.
\end{equation}
Thus \eqref{eq:connected-burnside}, the displayed recurrence, and \eqref{eq:scaled-magic-coordinates} reconstruct the CSV using integer and rational arithmetic alone.

The scaled tail $d!S_{d,r}(\cstar)$ belongs to $\Z[\sqrt3]$. Algorithm~\ref{alg:finite-certificate} stores this scaled quantity, whose sign is the sign of the original tail. Multiplication by $\cstar=2-\sqrt3$ sends
\[
 X+Y\sqrt3\longmapsto (2X-3Y)+(-X+2Y)\sqrt3
\]
for integers $X$ and $Y$. The sign test is completely integral. If $X$ and $Y$ have the same weak sign, the sign is immediate. If $X>0>Y$, then $X+Y\sqrt3>0$ exactly when $X^2>3Y^2$. If $Y>0>X$, then $X+Y\sqrt3>0$ exactly when $X^2<3Y^2$. Equality cannot occur unless $X=Y=0$, because $\sqrt3$ is irrational. Algorithm~\ref{alg:finite-certificate} uses this rule at every tail update. The tail checks can process one dimension at a time with arbitrary-precision integers. Reconstructing the entries uses the polynomial recurrence above as well.

\begin{algorithm}[H]
\caption{Verification of the exact finite certificate}\label{alg:finite-certificate}
\begin{algorithmic}[1]
\Require CSV entries $\alpha_{d,j}\in\Z$ for $1\le d\le1000$ and $0\le j<d$
\State Reject missing, repeated, out-of-range, or nonintegral entries
\State $\Delta_0(y)\gets1$
\For{$d\gets1$ \textbf{to} $1000$}
  \State Compute $B_d(y)$ from \eqref{eq:connected-burnside}
  \State $\Delta_d(y)\gets\sum_{k=1}^d\binom{d-1}{k-1}B_k(y)\Delta_{d-k}(y)$
  \For{$j\gets0$ \textbf{to} $d-1$}
    \State Compute $\widehat\alpha_{d,j}$ from \eqref{eq:scaled-magic-coordinates}
    \State Reject unless $\widehat\alpha_{d,j}=\alpha_{d,j}$
  \EndFor
  \If{$d\le3$}
    \State Reject unless the row is $(1)$, $(2,0)$, or $(6,0,6)$, respectively
  \Else
    \State Reject unless $\alpha_{d,0}\ge0$ and $\alpha_{d,d-1}\ge0$
    \State $Z\gets\alpha_{d,d-2}$; reject unless $Z>0$
    \For{$r\gets d-3$ \textbf{downto} $1$}
      \State $Z\gets\alpha_{d,r}+(2-\sqrt3)Z$
      \State Reject unless $Z>0$ by the exact integer sign test above
    \EndFor
  \EndIf
\EndFor
\State \Return \textsc{Pass}
\end{algorithmic}
\end{algorithm}

For each $4\le d\le1000$, this procedure certifies
\[
 a_{d,0}\ge0,
 \qquad a_{d,d-1}\ge0,
 \qquad S_{d,r}(2-\sqrt3)>0\quad(1\le r\le d-2).
\]
It checks exactly
\[
\sum_{d=4}^{1000}(d-2)=498500
\]
strict tail inequalities. For example, the first tail is
\[
 (d,r)=(4,1),
 \qquad S_{4,1}(2-\sqrt3)=\frac{156-76\sqrt3}{24}>0.
\]
Corollary~\ref{prop:cone} and the three displayed base cases prove Theorem~\ref{thm:intro-finite}. The accompanying data supplement consists of this single CSV, and Algorithm~\ref{alg:finite-certificate} supplies a language-independent reconstruction and verification procedure.

The table contains more information than the theorem uses. For every $4\le d\le1000$, all coordinates $a_{d,j}$ with $j\ne1$ are positive, while $a_{d,1}$ is positive for $d=4,6$, zero for $d=5$, and negative for $7\le d\le1000$; the only vanishing coordinates in the whole table are $a_{2,1}$, $a_{3,1}$, and $a_{5,1}$. Moreover, $50a_{d,1}+a_{d,2}>0$ for every $4\le d\le1000$, so the negative coordinate is small compared with the next one. In this range a single adjacent difference suffices, since
\[
 H_d(t)=a_{d,0}A_{d,1}(t)+\bigl(a_{d,1}+\cstar a_{d,2}\bigr)A_{d,2}(t)+a_{d,2}E^{(2)}_{d,\cstar}(t)+\sum_{j=3}^{d-1}a_{d,j}A_{d,j+1}(t)
\]
has nonnegative coefficients. The full weighted tails are needed in Section~\ref{sec:reverse}, where the estimates control sums of coordinates of small index but not individual coordinates.

\section{Reverse coordinates and the lower-index tails}\label{sec:reverse}

Beyond the finite certificate, a uniform argument is needed for all
sufficiently large dimensions. This section proves positivity of the weighted
tails in the lower-index regime and builds the reversed generating-function
identity needed for the upper regime. Small indices are handled in the
original, or forward, coordinates by a Cauchy estimate. Large indices become
more transparent after the coefficient order is reversed. Proposition~\ref{prop:low-r-block} proves the lower-index estimate for $d\ge2^{12}$. Proposition~\ref{prop:positive-reverse} in Section~\ref{sec:wlc} supplies positive coordinates for the remaining tails when $d\ge2^{62}$. Theorem~\ref{thm:WLC} combines these estimates through Corollary~\ref{prop:cone}, with the larger threshold determining the final range.

Write $\mathcal F_d(x)=[u^d]e^{K(u,x)}=\sum_jC_{d,j}x^j$ and put
\[
 \mathscr R_d(z)=\sum_{m=0}^{d-1}a_{d,d-1-m}z^m.
\]
Define
\begin{equation}\label{eq:P-kernel}
 \mathfrak P_n(z)=\frac1{n^2(n-1)!}\prod_{q=1}^{n-1}(2n-q+qz)
\end{equation}
and
\[
 J_n(z)=\sum_{h\mid n}\frac{\varphi(h)}{h^2}
 (z-1)^{n-n/h}\mathfrak P_{n/h}(z).
\]
The term with $h=1$ is $\mathfrak P_n(z)$, which has positive coefficients. The remaining divisor terms may have signs and will be bounded separately. Writing $J(u,z)=\sum_{n\ge1}J_n(z)u^n$ and
\[
 \Phi_z(w)=\frac{e^{(z-1)w}-1}{z-1},
 \qquad \Phi_1(w)=w,
\]
the transformation \eqref{eq:magic-coordinates} first gives
\[
 \mathcal A_d(w):=\sum_{j=0}^{d-1}a_{d,j}w^j
 =(1-w)^{d-1}\mathcal F_d\!\left(\frac{2w}{1-w}\right).
\]
Consequently
\[
 \mathscr R_d(z)=z^{d-1}\mathcal A_d(z^{-1})
 =(z-1)^{d-1}\mathcal F_d\!\left(\frac2{z-1}\right).
\]
To sum this identity over $d$, substitute $v=u(z-1)$ and
$y=2/(z-1)$ in $\exp K(v,y)$. If $n=ha$, the $h$-summand in
\eqref{eq:connected-burnside} satisfies
\[
 (z-1)^n\binom{(z+1)a/(z-1)-1}{a-1}
 =(z-1)^{n-a+1}a^2\mathfrak P_a(z).
\]
Since $a=n/h$, the coefficient of $u^n$ after substitution is therefore
\[
 \frac1{n^2}\sum_{h\mid n}\varphi(h)
 (z-1)^{n-a+1}a^2\mathfrak P_a(z)
 =\sum_{h\mid n}\frac{\varphi(h)}{h^2}
 (z-1)^{n-n/h+1}\mathfrak P_{n/h}(z)
 =(z-1)J_n(z).
\]
Hence $K(u(z-1),2/(z-1))=(z-1)J(u,z)$. Removing the constant term of the
exponential and dividing by $z-1$ now gives the formal identity
\begin{equation}\label{eq:reverse-identity}
 \sum_{d\ge1}\mathscr R_d(z)u^d=\Phi_z(J(u,z)).
\end{equation}
Set
\[
 \mathcal H(u,z)=e^{(z-1)J(u,z)}
 =\sum_{s\ge0}\mathcal H_s(z)u^s.
\]
Differentiating \eqref{eq:reverse-identity} in $u$ gives
\[
 \sum_{d\ge1}d\mathscr R_d(z)u^{d-1}
 =J_u(u,z)\mathcal H(u,z).
\]
In the coefficient of $u^{d-1}$, the term $J_{d-s}$ is multiplied by
$d-s$. Division by $d$ proves the exact marked identity
\begin{equation}\label{eq:marked-identity}
 \mathscr R_d(z)=\sum_{s=0}^{d-1}\left(1-\frac sd\right)
 J_{d-s}(z)\mathcal H_s(z).
\end{equation}

The marked identity makes the two regimes explicit. The proof separates the
tail index $r$ into two ranges. For $r\le d/4$, a Cauchy estimate in the
forward coordinates controls the whole weighted tail. For larger $r$, the
reverse identity above reduces the problem to the positivity of coefficients
near the opposite endpoint. We first treat the forward range, in which the
growth of $\mathcal F_d(x)$ is measured by
\[
 \mathscr E(x)=\frac{(1+x)^{1+x}}{x^x}\qquad(x>0).
\]
Since $\frac{d}{dx}\log\mathscr E(x)=\log\frac{1+x}{x}>0$, this function is increasing.

\begin{lemma}\label{lem:exponential-envelope}
Let $b_n\ge0$, $b_n\le Mn^{-5/2}$, and $A=\sum_{n\ge1}b_n<\infty$.
Then, for $n\ge1$,
\[
 [u^n]\exp\left(\sum_{j\ge1}b_ju^j\right)
 \le Me^A(A^3+6A^2+7A+1)n^{-5/2}.
\]
\end{lemma}

\begin{proof}
In a composition of $n$ into $k$ positive parts, one part is at least $n/k$.
Choose such a part and sum the remaining $k-1$ parts freely. The resulting
upper bound for the $k$-fold convolution is
$Mk^{7/2}A^{k-1}n^{-5/2}$.
Since $k^{7/2}\le k^4$,
\[
 \sum_{k\ge1}\frac{k^{7/2}A^{k-1}}{k!}
 \le\sum_{k\ge1}\frac{k^4A^{k-1}}{k!}=e^A(A^3+6A^2+7A+1).
\]
The last equality follows by applying $(A\,d/dA)^4$ to $e^A$ and dividing
by $A$, with the continuous interpretation at $A=0$.
\end{proof}

\begin{lemma}\label{lem:connected-envelope}
Let $x_0=5/54$ and $\ystar=\sqrt3-1$. For every $d\ge1$ and $x_0\le x\le2$,
\begin{equation}\label{eq:connected-upper}
 \mathcal F_d(x)\le e^{140}d^{-5/2}\mathscr E(x)^d,
\end{equation}
and
\begin{equation}\label{eq:connected-lower}
 \mathcal F_d(\ystar)\ge \frac1{36}d^{-5/2}\mathscr E(\ystar)^d.
\end{equation}
\end{lemma}

\begin{proof}
The coefficient $k_n(x)=[u^n]K(u,x)$ is the divisor sum \eqref{eq:connected-burnside} at $y=x$. Its term with $h=1$ is
\[
 \Lambda_n(x)=\frac1{n^2}\binom{(1+x)n-1}{n-1}
 =\frac1{n^2(n-1)!}\prod_{q=1}^{n-1}(xn+q),
\]
and the term with divisor $h$ is $\varphi(h)h^{-2}\Lambda_{n/h}(x)$. Since $t\mapsto\log(x+t)$ is increasing, its Riemann sums satisfy
\[
 n\!\int_0^1\log(x+t)\,dt-\log(1+x)
 \le\sum_{q=1}^{n-1}\log\Bigl(x+\frac qn\Bigr)
 \le n\!\int_0^1\log(x+t)\,dt-\log x.
\]
Robbins' inequalities \cite[equations~(1)--(2)]{Robbins1955}
\[
 \sqrt{2\pi N}(N/e)^Ne^{1/(12N+1)}<N!
 <\sqrt{2\pi N}(N/e)^Ne^{1/(12N)}
\]
combine with $n^{n-1}/(n-1)!=n^n/n!$ and
$\int_0^1\log(x+t)\,dt=\log\mathscr E(x)-1$ to give
\[
 \frac{e^{-1/(12n)}}{(1+x)\sqrt{2\pi}}
 \le\frac{\Lambda_n(x)}{n^{-5/2}\mathscr E(x)^n}
 \le\frac1{x\sqrt{2\pi}}.
\]
In particular, $1/(x\sqrt{2\pi})\le54/(5\sqrt{2\pi})<22$ for $x\ge x_0$, while
$1+\ystar<2$, $\sqrt{2\pi}<3$, and $e^{-1/12}>1/2$ give the lower bound; thus
\[
 \Lambda_n(x)\le22n^{-5/2}\mathscr E(x)^n,
 \qquad
 \Lambda_n(\ystar)\ge\frac1{36}n^{-5/2}\mathscr E(\ystar)^n.
\]
These estimates also hold at $n=1$, where $\Lambda_1(x)=1$.

The logarithmic series
\begin{equation}\label{eq:rational-log-certificate}
 \log a=2\sum_{j=0}^{N}\frac{v^{2j+1}}{2j+1}+R_N,
 \quad v=\frac{a-1}{a+1},
 \quad |R_N|\le
 \frac{2|v|^{2N+3}}{(2N+3)(1-v^2)}
\end{equation}
is an exact rational certificate whenever $a$ is rational. With $N=80$
at $x_0$ it gives $\log\mathscr E(x_0)>3/10$. Since
$1+3/10+(3/10)^2/2>64/49$, we have
$\mathscr E(x)\ge\mathscr E(x_0)>64/49$ for $x\ge x_0$.
Consequently, if $n=ha$ and $h\ge2$, then $n-a\ge n/2$ and
\[
 \frac{\varphi(h)}{h^2}\Lambda_a(x)
 \le22\mathscr E(x)^n n^{-5/2}h^{3/2}(7/8)^n.
\]
The continuous maximum of $t^{5/2}(7/8)^t$ is
$(5/(2e\log(8/7)))^{5/2}<200$: here
$e>27/10$, $\log(8/7)>1/8$, and the last comparison is the integer
inequality $200^3<27^5$. Since
$\sum_{h\mid n}h^{3/2}\le n^{5/2}$, the coefficients satisfy
\[
 0\le k_n(x)\le22(1+200)\mathscr E(x)^n n^{-5/2}<5000\mathscr E(x)^n n^{-5/2}.
\]

The normalized coefficients also have a bounded total mass.
Indeed, writing $\vartheta_x=\mathscr E(x)^{-1}<49/64$ and $n=ha$,
\[
 \sum_{n\ge1}k_n(x)\mathscr E(x)^{-n}
 =\sum_{a\ge1}\Lambda_a(x)\mathscr E(x)^{-a}
 \sum_{h\ge1}\frac{\varphi(h)}{h^2}\vartheta_x^{(h-1)a}
 <22\cdot\frac32\cdot3=99.
\]
Here the inner sum is at most
$1+\vartheta_x/(2(1-\vartheta_x))<3$, and $\sum_{a\ge1}a^{-5/2}<3/2$.
Lemma~\ref{lem:exponential-envelope}, applied to $b_n=k_n(x)\mathscr E(x)^{-n}$ with $M=5000$ and $A<99$, now gives
\[
 \mathcal F_d(x)\mathscr E(x)^{-d}
 <5000e^{99}(100^3+6\cdot100^2+7\cdot100+1)d^{-5/2}
 <e^{140}d^{-5/2}.
\]
For the last comparison, the factor outside $e^{99}$ is below $2^{34}$,
and $2<e$.
This proves \eqref{eq:connected-upper}. Positivity of the connected
coefficients and the summand $\Lambda_d(\ystar)$ corresponding to $h=1$ prove
\eqref{eq:connected-lower}.
\end{proof}

\begin{lemma}\label{lem:low-endpoint-bounds}
Let $\cstar=2-\sqrt3$, $\ystar=\sqrt3-1$, $\varrho=5/113$, and $x_0=5/54$. Then
\begin{align}
 \left(\frac{\cstar}{\varrho}\right)^{1/4}\frac{1+\varrho}{\ystar}
 \frac{\mathscr E(x_0)}{\mathscr E(\ystar)}&<\frac{19}{20},
 \label{eq:BR}\\
 \frac{\cstar}{\ystar}\frac{\mathscr E(2)}{\mathscr E(\ystar)}&<\frac{77}{100}.
 \label{eq:BU}
\end{align}
\end{lemma}

\begin{proof}
The integer comparisons $1732^2<3\cdot10^6<1733^2$ give
$183/250<\ystar<733/1000$ and $\cstar<268/1000$. Formula \eqref{eq:rational-log-certificate},
with $N=80$ at $5/54$ and $N=20$ at $183/250$, gives
\[
 \frac3{10}<\log\mathscr E(5/54)<\frac{3172}{10000},
 \qquad
 \log\mathscr E(183/250)>\frac{11797}{10000}.
\]
Since $\mathscr E$ is increasing,
\[
 \frac{\mathscr E(x_0)}{\mathscr E(\ystar)}<\frac{423}{1000},
\]
because $3172/10000-11797/10000=-69/80$ and
\[
 \sum_{j=0}^4\frac{(69/80)^j}{j!}>\frac{1000}{423}.
\]
Also
\[
 (\cstar/\varrho)^{1/4}<\frac{157}{100},\qquad
 \frac{1+\varrho}{\ystar}<\frac{143}{100};
\]
with the rational bounds for $\cstar$ and $\ystar$, the two residual differences are respectively
$1893201/10^8$ after taking fourth powers and $7097/2067900$.
Now
\[
 \frac{157}{100}\frac{143}{100}\frac{423}{1000}
 =\frac{9496773}{10^7}<\frac{19}{20}.
\]
Finally the same logarithmic lower bound and the degree-six exponential
sum give $\mathscr E(\ystar)>13/4$. Since $\cstar/\ystar=\ystar/2<733/2000$ and
$\mathscr E(2)=27/4$, the left side of \eqref{eq:BU} is below
$98955/130000<77/100$.
\end{proof}

\begin{proposition}\label{prop:low-r-block}
For every $d\ge2^{12}$,
\[
 S_{d,r}(\cstar)>0\qquad(1\le r\le\lfloor d/4\rfloor).
\]
\end{proposition}

\begin{proof}
Use the constants of Lemma~\ref{lem:low-endpoint-bounds}. By \eqref{eq:magic-coordinates}, the generating polynomial of the coordinates is
\[
 \mathcal A_d(z)=\sum_ja_{d,j}z^j
 =\sum_j C_{d,j}(2z)^j(1-z)^{d-1-j}
 =(1-z)^{d-1}\mathcal F_d\!\left(\frac{2z}{1-z}\right).
\]
At $z=\cstar$ we have $1-\cstar=\ystar$ and $2\cstar/(1-\cstar)=\ystar$. Hence the weighted sum including the top coordinate,
\[
 T_{d,r}=\sum_{j=r}^{d-1}\cstar^{j-r}a_{d,j},
\]
satisfies
\[
 T_{d,r}=\mathcal M_{d,r}-\cstar^{-r}\sum_{j<r}\cstar^ja_{d,j},
 \qquad \mathcal M_{d,r}=\cstar^{-r}\mathcal A_d(\cstar)
 =\cstar^{-r}\ystar^{d-1}\mathcal F_d(\ystar)>0.
\]
We bound the low coordinates by Cauchy's estimate on the circle $|z|=\varrho$. The middle expression for $\mathcal A_d$ and the positivity of the $C_{d,j}$ give
\[
 |\mathcal A_d(z)|
 \le(1+\varrho)^{d-1}\mathcal F_d\!\left(\frac{2\varrho}{1+\varrho}\right)
 \le(1+\varrho)^{d-1}\mathcal F_d(x_0)\qquad(|z|=\varrho),
\]
because $2\varrho/(1+\varrho)=5/59<x_0$ and $\mathcal F_d$ has
nonnegative coefficients. Cauchy's estimate $|a_{d,j}|\le\varrho^{-j}\max_{|z|=\varrho}|\mathcal A_d(z)|$ and the geometric sum
\[
 \sum_{j=0}^{r-1}(\cstar/\varrho)^j
 <\frac{\varrho}{\cstar-\varrho}(\cstar/\varrho)^r
\]
therefore give
\[
 \Bigl|\cstar^{-r}\sum_{j<r}\cstar^ja_{d,j}\Bigr|
 \le\mathcal M_{d,r}\varepsilon_{d,r},
 \qquad
 \varepsilon_{d,r}=\frac{\varrho}{\cstar-\varrho}\left(\frac{\cstar}{\varrho}\right)^r
 \left(\frac{1+\varrho}{\ystar}\right)^{d-1}
 \frac{\mathcal F_d(x_0)}{\mathcal F_d(\ystar)}.
\]
The tail $S_{d,r}(\cstar)$ omits the top term of $T_{d,r}$. Formula \eqref{eq:magic-coordinates} and the positivity of the $C_{d,j}$ give $|a_{d,d-1}|\le\mathcal F_d(2)$, and hence
\[
 \cstar^{d-1-r}a_{d,d-1}\le\mathcal M_{d,r}\varepsilon'_d,
 \qquad
 \varepsilon'_d=\left(\frac{\cstar}{\ystar}\right)^{d-1}
 \frac{\mathcal F_d(2)}{\mathcal F_d(\ystar)}.
\]
Lemmas \ref{lem:connected-envelope} and \ref{lem:low-endpoint-bounds}
imply, for $r\le d/4$,
\[
 \varepsilon_{d,r}\le36e^{140}(19/20)^d
 \frac{\varrho}{\cstar-\varrho}\frac{\ystar}{1+\varrho},
 \qquad
 \varepsilon'_d\le36e^{140}(77/100)^d\frac{\ystar}{\cstar}.
\]
Here $\cstar>1/4$, $\varrho<1/20$, and $\ystar<1+\varrho$, so the product of the last two factors in the first bound is below $1/4$. Also $\ystar/\cstar<3$, because $\ystar<3\cstar$ is equivalent to $4\sqrt3<7$, which follows from $3<49/16$. Therefore
\[
 \varepsilon_{d,r}<36e^{140}(19/20)^d,
 \qquad
 \varepsilon'_d<108e^{140}(77/100)^d.
\]
Since $(19/20)^{14}<1/2$, $(77/100)^3<1/2$, and $e<4$,
the first quantity at $d=2^{12}$ is below
$2^{6+280-\lfloor4096/14\rfloor}=2^{-6}$, and the second is below
$2^{7+280-\lfloor4096/3\rfloor}<1/4$.
Both decrease with $d$. Finally
\[
 S_{d,r}(\cstar)=T_{d,r}-\cstar^{d-1-r}a_{d,d-1}
 \ge\mathcal M_{d,r}(1-\varepsilon_{d,r}-\varepsilon'_d)>0.
\]
\end{proof}

\section{Positivity of the reversed coordinates}\label{sec:wlc}

The remaining tails use coefficients near the opposite end of the refined
expansion. We prove that all of these coefficients are positive by comparing
the marked identity with a positive exponential series. The term $\mathfrak P_n$ corresponding to $h=1$ in \eqref{eq:marked-identity} has positive coefficients. We choose a positive parameter so that the target coefficient occurs at the mean of an associated Bernoulli sum. For $n$ close to $d$, the corresponding point probabilities vary slowly; for smaller $n$, absolute coefficient bounds suffice. Summing the signed corrections leaves the positive factor $\mathcal H^\star(x)$ defined below. Their total absolute mass and first moments control the error. The comparison uses Darroch's mode theorem \cite[Theorem~4]{Darroch1964} and an elementary Fourier estimate for adjacent probabilities, uniformly down to the first reversed coefficient.

For $x>0$ and $0\le t\le1$, put $f_x(t)=xt/(2+(x-1)t)$. For every integer $n\ge1$, define
\[
 p_{n,q}(x)=f_x(q/n)=\frac{qx}{2n-q+qx},\qquad
 X_{n,x}=\sum_{q=1}^{n-1}\operatorname{Bernoulli}(p_{n,q}(x)),
\]
with independent summands, and let $\mu_n(x)$ and $V_n(x)$ denote its mean
and variance. The product in \eqref{eq:P-kernel} gives
\[
 [z^m]\mathfrak P_n(z)x^m=\mathfrak P_n(x)\Pr(X_{n,x}=m),\qquad
 \frac{\mu_n(x)}{1+x}\le V_n(x)\le\mu_n(x).
\]
For $n\ge2$, the mean increases continuously and strictly from $0$ to $n-1$.
Thus $\mu_d(x)=m$ determines a unique positive $x$ whenever
$1\le m\le d-2$.

\subsection{Bounds for the correction series}

The series $\mathcal H$ in \eqref{eq:marked-identity} has coefficients of both signs. We bound their total absolute mass and first moments so that the effect of each shift of a coefficient of $\mathfrak P_n$ can be summed. For a polynomial $F$, write $\|F\|_x=\sum_j|[z^j]F|x^j$.
Put
\[
 \lambda(x)=\exp\left(1+\int_0^1\log(2+(x-1)t)\,dt\right).
\]

\begin{lemma}\label{lem:primitive-wiener}
For $0\le x\le15$ and $n\ge2$,
\begin{equation}\label{eq:primitive-product}
 \mathfrak P_n(x)=\frac{\lambda(x)^n}{\sqrt{4\pi(1+x)}\,n^{5/2}}
 e^{\delta_n(x)},\qquad -\frac1n<\delta_n(x)<0.
\end{equation}
Moreover $\lambda(x)\ge4$, $(1+x)/\lambda(x)\le4/5$, and
\begin{equation}\label{eq:periodic-bound}
 \|J_n-\mathfrak P_n\|_x\le\lambda(x)^n(4/5)^{n/2}\qquad(n\ge1).
\end{equation}
\end{lemma}

\begin{proof}
For the concave function $\psi(t)=\log(2+(x-1)t)$, the trapezoidal estimate gives
\[
 \sum_{q=1}^{n-1}\psi(q/n)
 =n\int_0^1\psi(t)\,dt-\frac{\psi(0)+\psi(1)}2-\epsilon_n,\qquad
 0\le \epsilon_n\le\frac{(x-1)^2}{16n(1+x)}\le\frac{49}{64n}.
\]
For completeness, on an interval $[t_0,t_1]$ the integral of the gap above
the chord is at most $(t_1-t_0)^2(\psi'(t_0)-\psi'(t_1))/8$; summing this bound over
the $n$ equal subintervals gives the display.
Robbins' formula \cite[equations~(1)--(2)]{Robbins1955},
\[
 n!=\sqrt{2\pi n}(n/e)^ne^{\rho_n},\qquad
 \frac1{12n+1}<\rho_n<\frac1{12n},
\]
and $n^{n-1}/(n-1)!=n^n/n!$ give
$\delta_n=-\rho_n-\epsilon_n$. Since $1/12+49/64<1$, this proves
\eqref{eq:primitive-product}.

The function $\lambda$ is increasing and $\lambda(0)=4$. Direct integration gives
\[
 \log\frac{1+x}{\lambda(x)}
 =-\frac{2\log((1+x)/2)}{x-1}.
\]
The quotient on the right, before the minus sign, is positive and decreasing,
with its continuous value at $x=1$. At $x=15$ its exponential is
$8^{1/7}>5/4$, since $8\cdot4^7>5^7$. This proves the ratio bound.
For a divisor $h\ge2$ of $n$, put $a=n/h$. Positivity of $\mathfrak P_a$ gives
\[
 \|(z-1)^{n-a}\mathfrak P_a(z)\|_x
 \le(1+x)^{n-a}\mathfrak P_a(x)
 \le\lambda(x)^n a^{-5/2}(4/5)^{n/2}.
\]
The bound $\mathfrak P_a(x)\le\lambda(x)^aa^{-5/2}$ also holds at $a=1$.
Multiply by $\varphi(h)/h^2$, use $\varphi(h)\le h$, and sum
$h^{3/2}\le n^{5/2}$ over divisors by bounding their number by $n$.
This proves \eqref{eq:periodic-bound}; at $n=1$ its left side is zero.
\end{proof}

The useful information about $\mathcal H$ consists of its total norm and
its first moments. These remain small even though its coefficients have signs.

\begin{lemma}\label{lem:correction-moments}
For $0\le x\le15$, set
\[
 \widehat h_s(x)=\|\mathcal H_s\|_x\lambda(x)^{-s},\qquad
 h_{s,k}(x)=[z^k]\mathcal H_s(z)x^k\lambda(x)^{-s}.
\]
Then
\begin{align}
 \sum_{s\ge0}\widehat h_s(x)&<16,&
 \sum_{s\ge0}s\widehat h_s(x)&<176,\label{eq:H-moments}\\
 \sum_{s,k\ge0}k|h_{s,k}(x)|&\le\frac{176x}{1+x},&
 \widehat h_s(x)&<2^{16}s^{-5/2}\quad(s\ge1).\label{eq:H-degree-point}
\end{align}
Also
\begin{equation}\label{eq:J-first-moment}
 \sum_{n\ge1}n\|J_n\|_x\lambda(x)^{-n}<\frac{11}{1+x},
\end{equation}
and the real number
\[
 \mathcal H^\star(x):=\mathcal H(\lambda(x)^{-1},x)
 =\exp\bigl((x-1)J(\lambda(x)^{-1},x)\bigr)
\]
satisfies $\mathcal H^\star(x)>2/3$.
\end{lemma}

\begin{proof}
Introduce the power series with nonnegative coefficients
\[
 \mathcal B(u,z)=(1+z)\sum_{n\ge1}u^n
 \sum_{h\mid n}\frac{\varphi(h)}{h^2}
 (1+z)^{n-n/h}\mathfrak P_{n/h}(z).
\]
Its exponential majorizes the absolute values of the coefficients of
$\mathcal H(u,z)$. Write $R=\lambda(x)^{-1}$,
$\nu=(1+x)R\le4/5$, and
\[
 v_a=(1+x)\mathfrak P_a(x)R^a,\qquad
 \beta=\mathcal B(R,x),\qquad
 \beta_1=R\mathcal B_u(R,x).
\]
Here $v_1=\nu$, and \eqref{eq:primitive-product} of Lemma~\ref{lem:primitive-wiener}, together with
$\pi>3$, gives $v_a<(6/5)a^{-5/2}$ for $a\ge2$.
The latter bound also holds for $a=1$.
Changing variables from $n,h$ to $a,h$, where $n=ha$, yields
\[
 \beta=\sum_{a\ge1}v_a\sum_{h\ge1}
 \frac{\varphi(h)}{h^2}\nu^{(h-1)a}.
\]
The contribution from $a=1$ is at most $-\log(1-\nu)\le\log5<13/8$.
For $a\ge2$, the inner sum is at most
$1+\nu^a/(2(1-\nu^a))\le17/9$.
The elementary sums
\[
 \sum_{a\ge2}a^{-5/2}<\frac38,\qquad
 \sum_{a\ge2}a^{-3/2}<2
\]
follow by separating the first terms and integrating the tails.
Consequently
\[
 \beta<\frac{13}{8}+\frac65\frac38\frac{17}{9}
 =\frac{99}{40}<\frac52,\qquad
 \beta_1<4+\frac65\frac{25}{9}\,2<11.
\]
For the second estimate use $\varphi(h)/h\le1$; the $a=1$ contribution
is at most $\nu/(1-\nu)\le4$, and $1-\nu^a\ge9/25$ for $a\ge2$.
The bound $\log5<13/8$ is certified by
$\sum_{j=0}^5(13/8)^j/j!>5$.
For the first tail sum above,
$2^{-5/2}+3^{-5/2}+\int_3^\infty t^{-5/2}dt
=1/(4\sqrt2)+1/(3\sqrt3)<3/8$; the second follows from
$2^{-3/2}+\int_2^\infty t^{-3/2}dt<2$.

It follows that the total mass of $e^{\mathcal B(Ru,x)}$ is
$e^\beta<e^{5/2}<16$, and its first $u$-moment is
$e^\beta\beta_1<176$. To control its $z$-moment, hold $R$ fixed and observe that the
logarithmic derivative of every polynomial
$(1+z)^{n-a+1}\mathfrak P_a(z)$ at $z=x$ is at most $n/(1+x)$.
Indeed every factor $2a-j+jz$ of $\mathfrak P_a$ has logarithmic derivative
$j/(2a-j+jx)\le1/(1+x)$ at $z=x$.
Thus $x\mathcal B_z(R,x)\le x\beta_1/(1+x)$, proving the first
three estimates. Here $e^{5/2}<16$ follows from $e<3$ and
$9\sqrt3<16$. The bound for $\beta_1$ also proves
\eqref{eq:J-first-moment}.

For the pointwise bound let $b_n=[u^n]\mathcal B(Ru,x)$.
The term with $h=1$ is at most $(6/5)n^{-5/2}$. For the remaining divisors,
$\varphi(h)/h^2\le1/h=a/n$ gives
\[
 n^{5/2}b_n
 \le\frac65+n^{3/2}(4/5)^{n/2}\sum_{a\ge1}av_a
 <\frac65+12\frac{18}{5}<45.
\]
Here $\sum_{a\ge1}a^{-3/2}<3$, and maximization over positive real $t$ gives
\[
 \sup_{t>0}t^{3/2}(4/5)^{t/2}
 =\left(\frac3{e\log(5/4)}\right)^{3/2}
 <\left(\frac{81}{16}\right)^{3/2}=\frac{729}{64}<12.
\]
We used $e>8/3$ and $\log(5/4)>2/9$.
Apply Lemma~\ref{lem:exponential-envelope} with $M=45$ and $A=\beta<5/2$.
It gives
\[
 [u^s]e^{\mathcal B(Ru,x)}
 <45\cdot16\left((5/2)^3+6(5/2)^2+7(5/2)+1\right)s^{-5/2}
 =51570s^{-5/2}<2^{16}s^{-5/2}.
\]

It remains to bound the real factor $\mathcal H^\star(x)$ from below. If $1\le x\le15$,
every summand in $J(R,x)$ is nonnegative, so $\mathcal H^\star(x)\ge1$. For $0\le x\le1$,
the divisor expansion gives
\[
 |(x-1)J(R,x)|\le G(x),
\]
where
\[
 G(x)=\sum_{a\ge1}(1-x)\mathfrak P_a(x)R^a
 \sum_{h\ge1}\frac{\varphi(h)}{h^2}
 \bigl((1-x)R\bigr)^{(h-1)a}.
\]
We show that $G(x)\le G(0)$. Put
\[
 I(x)=\int_0^1f_x(t)\,dt=x\frac{\lambda'(x)}{\lambda(x)}.
\]
Since $f_x$ is increasing and $f_x(0)=0$, its left Riemann sum gives
$\mu_a(x)\le aI(x)$. Thus, for $0<x<1$,
\[
 \frac{d}{dx}\log\bigl(\mathfrak P_a(x)R^a\bigr)
 =\frac{\mu_a(x)-aI(x)}x\le0.
\]
The factors $1-x$ and $(1-x)R$ also decrease. Each nonnegative summand
of $G(x)$ is therefore nonincreasing, and continuity of the factors
extends the comparison to the endpoints.

At $x=0$, we have $R=1/4$ and $\mathfrak P_1(0)=1$. The $a=1$ contribution is at most
\[
 \frac14+\frac1{64}+\frac{(1/4)^3}{3(1-1/4)}
 =\frac{157}{576},
\]
using $\varphi(1)=\varphi(2)=1$ and $\varphi(h)/h^2\le1/3$ for $h\ge3$.
For $a\ge2$, \eqref{eq:primitive-product} and $\pi>3$ give
\[
 \mathfrak P_a(0)4^{-a}<\frac{29}{100}a^{-5/2},\qquad
 \sum_{h\ge1}\frac{\varphi(h)}{h^2}4^{-(h-1)a}
 \le1+\frac{4^{-a}}{2(1-4^{-a})}\le\frac{31}{30}.
\]
Here $1/\sqrt{12}<29/100$ follows by squaring. Using the bound
$\sum_{a\ge2}a^{-5/2}<3/8$ established above, we obtain
\[
 G(x)\le G(0)<\frac{157}{576}+\frac{29}{100}\frac38\frac{31}{30}
 =\frac{157}{576}+\frac{899}{8000}<\frac25.
\]
Finally, the series
$\log(3/2)=2\sum_{j\ge0}(1/5)^{2j+1}/(2j+1)>2/5$ yields
$\mathcal H^\star(x)\ge e^{-G(x)}>2/3$. This proves the remaining assertion.
\end{proof}

\subsection{Comparing probabilities by discrete smoothing}

We compare two Bernoulli sums by replacing their parameters one at a time. Each replacement changes a point probability by a parameter difference times an adjacent probability difference. A Fourier bound controls the latter uniformly, and Darroch's theorem supplies the lower bound needed to turn the comparison into a ratio.

\begin{lemma}\label{lem:pb-smoothing}
Let $d\ge16$, $0<x\le15$, and $\mu_d(x)=m$ be an integer with
$1\le m\le d-2$. Let $s$ and $k$ be nonnegative integers with $n=d-s\ge d/2$. Then
\begin{equation}\label{eq:smoothing-ratio}
 \left|\frac{\Pr(X_{n,x}=m-k)}{\Pr(X_{d,x}=m)}-1\right|
 \le\frac{1792(3sx+k)}{\sqrt{m+1}}.
\end{equation}
\end{lemma}

\begin{proof}
If $Y$ is a sum of independent Bernoulli variables with variance $V$, Fourier
inversion gives, for $V>0$,
\[
 \sup_j|\Pr(Y=j-1)-\Pr(Y=j)|
 \le\frac1\pi\int_0^\pi t e^{-2Vt^2/\pi^2}\,dt
 =\frac{\pi(1-e^{-2V})}{4V}<\frac1V.
\]
We used $|1-e^{it}|\le|t|$,
$|\mathbb Ee^{itY}|\le e^{-2V\sin^2(t/2)}$, and
$\sin(t/2)\ge t/\pi$ on $[0,\pi]$. Combining this with the trivial bound
$1$, including $V=0$, proves
\begin{equation}\label{eq:adjacent-mass}
 \sup_j|\Pr(Y=j-1)-\Pr(Y=j)|\le\frac2{1+V}.
\end{equation}

We compare $X_{d,x}$ with $X_{n,x}$ by replacing the first $n-1$
parameters and deleting the last $s$ summands. Put
$p_*=f_x(1)=x/(1+x)$.
Every retained parameter lies between $f_x(q/d)$ and $f_x(q/n)$
and is at most $p_*$. Each replacement uses the law of all the other
summands. Such a law retains, with at most one exception, the parameters
indexed by $1\le q\le N=\lceil d/2\rceil-1$.
Since $q/d\le1/2$,
\[
 f_x(q/d)\ge\frac{xq}{9d}.
\]
If $\theta$ is any of these retained parameters, then
$\theta(1-\theta)\ge p_*q/(9d)$. Therefore every intermediate Bernoulli sum has variance at least the following bound, which also holds for $X_{n,x}$:
\[
 \frac{p_*N(N-1)}{18d}
 \ge\frac{21p_*d}{2304}>\frac{m}{128}.
\]
Here $d\ge16$ gives
$N(N-1)\ge(d/2-1)(d/2-2)\ge21d^2/128$, and
$m=\mu_d(x)\le(d-1)p_*<dp_*$.
By \eqref{eq:adjacent-mass}, the adjacent probability difference of every
such law is at most $256/(m+1)$.

Changing a Bernoulli parameter from $p$ to $p'$ changes a point probability
by at most $|p-p'|$ times the adjacent difference of the remaining law.
Since $|f_x'|\le2x$, the sum of the parameter changes and deletions is
at most $3sx$. Shifting the argument by $k$ costs at most $k$ further
adjacent differences. Hence
\[
 |\Pr(X_{n,x}=m-k)-\Pr(X_{d,x}=m)|
 \le\frac{256(3sx+k)}{m+1}.
\]

Darroch's mode theorem \cite[Theorem~4]{Darroch1964} states that,
for a sum $Y$ of independent Bernoulli variables, any integer $q$
maximizing $\Pr(Y=q)$ satisfies $|q-\mathbb EY|<1$.
Pitman states the theorem in this form, together with Darroch's finer
description of the mode, in his survey of real-rooted generating polynomials
\cite[p.~284]{Pitman1997}. Since $X_{d,x}$ has integer mean
$m$, this result makes $m$ its unique mode, so $\Pr(X_{d,x}=m)$ is its
largest point probability. Chebyshev's inequality assigns mass greater than
$3/4$ to the interval of radius $2\sqrt{V_d+1}$ around $m$.
It contains at most $5\sqrt{m+1}$ integers, because $V_d\le m$.
Consequently
\begin{equation}\label{eq:modal-lower}
 \Pr(X_{d,x}=m)>\frac1{7\sqrt{m+1}}.
\end{equation}
Dividing the preceding difference bound by this lower bound proves
\eqref{eq:smoothing-ratio}.
\end{proof}

\subsection{Positivity of the reversed coefficients}

We now choose the Bernoulli parameter so that the coefficient of interest occurs at the mean. The preceding bounds compare each term of \eqref{eq:marked-identity} with the contribution from $\mathfrak P_{d-s}$ and leave an error smaller than the positive factor $\mathcal H^\star(x)$.

\begin{proposition}\label{prop:positive-reverse}
If $d\ge2^{62}$ and $1\le m\le d-2-\lfloor d/4\rfloor$, then
$a_{d,d-1-m}>0$.
\end{proposition}

\begin{proof}
Choose $x$ with $\mu_d(x)=m$, and put $\kappa=[z^m]\mathfrak P_d(z)>0$.
This parameter satisfies $x<15$. Indeed, increasing-function Riemann bounds give
\[
 \mu_d(15)\ge d\int_0^1f_{15}(t)\,dt-f_{15}(1),\qquad
 \int_0^1f_{15}(t)\,dt
 =\frac{15}{196}(14-2\log8)>\frac{751}{1000}.
\]
The last inequality follows from $\log8<209/100$, certified by the
degree-six exponential sum. For $d>1000$ this gives
$\mu_d(15)>3d/4>m$. Also
$\mu_d(x)\ge x(d-1)/32$, so
\[
 \frac{x}{\sqrt{m+1}}<\frac{64}{\sqrt d}.
\]

For integers $0\le s\le d/2$, put $n=d-s$ and define
\[
 \omega_{d,s}=\left(1-\frac sd\right)\lambda(x)^s\frac{\mathfrak P_n(x)}{\mathfrak P_d(x)}
 =\left(\frac d{d-s}\right)^{3/2}e^{\delta_n-\delta_d}.
\]
Equation \eqref{eq:primitive-product} implies
\begin{equation}\label{eq:A-normalization}
 \omega_{d,s}<4,\qquad |\omega_{d,s}-1|<\frac{32(s+1)}d.
\end{equation}
To see this directly, the derivative of $(1-t)^{-3/2}$ is below $9$ on
$0\le t\le1/2$, its value is below $3$, and
$|\delta_n-\delta_d|<2/d$.
Thus $|e^{\delta_n-\delta_d}-1|<4/d$ and
$|\omega_{d,s}-1|<(9s+12)/d$. The upper bound follows from
$\omega_{d,s}<3e^{1/d}<4$.

Write
\[
 \chi_{s,k}=\frac{\Pr(X_{n,x}=m-k)}{\Pr(X_{d,x}=m)}.
\]
After division by $\kappa$, the part of the $s$-summand that contains $\mathfrak P_{d-s}$ in
\eqref{eq:marked-identity} is
\[
 \omega_{d,s}\sum_k h_{s,k}(x)\chi_{s,k}.
\]
Compare the marked identity with $\mathcal H^\star(x)=\sum_{s,k}h_{s,k}(x)$.
We estimate the $h=1$ terms for $s\le d/2$, the full terms for
$s>d/2$, the omitted tail of $\mathcal H^\star$, and the terms with $h\ge2$ for $s\le d/2$.
These four parts account for every term.

For the first part, \eqref{eq:A-normalization},
Lemma~\ref{lem:pb-smoothing}, and the first-moment bounds of Lemma~\ref{lem:correction-moments} give
\begin{align*}
 &\left|\sum_{s\le d/2}\sum_k h_{s,k}(\omega_{d,s}\chi_{s,k}-1)\right|\\
 &\quad<
 \frac{32(176+16)}d+
 \frac{4\cdot1792}{\sqrt{m+1}}
 \left(3x\cdot176+\frac{176x}{1+x}\right)\\
 &\quad<
 \frac{6144}d+\frac{322961408}{\sqrt d}.
\end{align*}
For the second part, use the full $J_n$ term with $n=d-s<d/2$.
The product satisfies $\mathfrak P_d(x)>\lambda(x)^dd^{-5/2}/32$,
and \eqref{eq:modal-lower} gives a reciprocal probability below
$7\sqrt d$. Since $s^{-5/2}<6d^{-5/2}$,
\[
 \sum_{s>d/2}
 \left|\frac{(1-s/d)[z^m](J_{d-s}\mathcal H_s)}{\kappa}\right|
 <\frac{1344\cdot2^{16}}{\sqrt d}
 \sum_{n\ge1}n\|J_n\|_x\lambda(x)^{-n}
 <\frac{968884224}{\sqrt d}.
\]
The third part is bounded by
\[
 \sum_{s>d/2}\widehat h_s
 <2^{16}\sum_{s>d/2}s^{-5/2}
 <2^{18}d^{-3/2}.
\]
For the last step, integration from $\lfloor d/2\rfloor\ge d/3$
gives a bound $2\sqrt3\,d^{-3/2}<4d^{-3/2}$.

Finally, \eqref{eq:periodic-bound} bounds the terms with $h\ge2$ by
\[
 224d^3\sum_{s\le d/2}(4/5)^{(d-s)/2}\widehat h_s
 <2^{12}d^3e^{-d/20}.
\]
Here $(4/5)^{d/4}<e^{-d/20}$ follows from $\log(5/4)>1/5$.
Combining the four parts, and using
$6144+322961408+968884224+2^{18}=1292113920<5\cdot2^{28}$,
yields the single estimate
\[
 \left|\frac{a_{d,d-1-m}}{\kappa}-\mathcal H^\star(x)\right|
 <\frac{5\cdot2^{28}}{\sqrt d}+2^{12}d^3e^{-d/20}.
\]
For $d\ge2^{62}$ the first term is at most $5/8$.
The second decreases for $d>60$ and is below $1/64$ at $2^{62}$:
$e>2$ bounds it by $2^{198-2^{62}/20}<2^{-6}$.
Since $\mathcal H^\star(x)>2/3$, the normalized coefficient is greater than
$2/3-5/8-1/64=5/192>0$.
\end{proof}

\begin{table}[ht]
\centering
\caption{The two estimates cover every weighted tail. A tail with
$r>\lfloor d/4\rfloor$ uses exactly the reverse indices
$1\le m\le d-1-r$, all covered by the second row.}
\label{tab:analytic-ranges}
\small
\begin{tabularx}{\textwidth}{@{}>{\raggedright\arraybackslash}p{.29\textwidth}>{\raggedright\arraybackslash}p{.22\textwidth}>{\raggedright\arraybackslash}X@{}}
\toprule
Index range & Sufficient bound on $d$ & Positive quantity and method\\
\midrule
$1\le r\le\lfloor d/4\rfloor$ & $d\ge2^{12}$ &
$S_{d,r}(c_*)$, by the forward Cauchy bound\\
$1\le m\le d-2-\lfloor d/4\rfloor$ & $d\ge2^{62}$ &
$a_{d,d-1-m}$, by the uniform smoothing estimate\\
\bottomrule
\end{tabularx}
\end{table}

\begin{samepage}
\begin{theorem}\label{thm:WLC}
For every integer $d\ge2^{62}$ and every $1\le r\le d-2$,
$S_{d,r}(\cstar)>0$. Consequently all zeros of $H_d(t)$ are real and negative
for every $d\ge2^{62}$.
\end{theorem}

\begin{proof}
Table~\ref{tab:analytic-ranges} records the two ranges. Proposition~\ref{prop:low-r-block} covers $r\le\lfloor d/4\rfloor$.
For $r>\lfloor d/4\rfloor$, put $m_r=d-1-r$. Reversing the tail gives
\[
 S_{d,r}(\cstar)=\sum_{m=1}^{m_r}\cstar^{m_r-m}a_{d,d-1-m}.
\]
Every coefficient in this sum is positive by
Proposition~\ref{prop:positive-reverse}. Since $\cstar>0$,
Corollary~\ref{prop:cone} completes the proof.
\end{proof}
\end{samepage}

\section{Conclusion and the open range}\label{sec:frontiers}

The real-rootedness question for $P_d(132,213)$ leads to a difficulty that its cube structure and volume do not reveal: its expansion in the refined Eulerian basis has a negative coordinate already at $d=7$. The compatible family of adjacent differences overcomes this obstruction whenever the weighted tails \eqref{eq:tails} are nonnegative. The finite certificate and the coefficient estimates establish this condition for $d\le1000$ and $d\ge2^{62}$, respectively, giving a partial answer to the unimodality question for this family.

The criterion applies to any polynomial with the required signs in these coordinates, so its use extends beyond nonnegative expansions in the refined basis. The constant $\cstar=2-\sqrt3$ is optimal for the chosen family simultaneously in all dimensions $d\ge4$; starting at $d=5$ permits a larger parameter. For the geometric model, the explicit lattice maps connect $P_d$ to the known chain and discriminant polytopes. They identify its lattice points with those of $2R_d$, while the lattice points of $R_d$ are ordinary tournament score sequences. The interior-point shift then explains why the top coefficient of $H_d$ counts strong score sequences for $d\ge3$.

The interval $1001\le d<2^{62}$ remains open. The finite data and the eventual result support the following conjecture, whose resolution requires control in every intermediate dimension.
\begin{conjecture}\label{conj:uniform}
For every $d\ge3$, all zeros of $H_d(t)$ are real and negative.
\end{conjecture}

The data suggest a specific route to this conclusion. For $7\le d\le1000$, the only negative coordinate is $a_{d,1}$, and $50a_{d,1}+a_{d,2}>0$. Proving $a_{d,j}\ge0$ for $2\le j\le d-1$ and $a_{d,1}+\cstar a_{d,2}\ge0$ for every $d\ge4$ would settle Conjecture~\ref{conj:uniform} through the single adjacent difference $E^{(2)}_{d,\cstar}$, as in Section~\ref{sec:finite}. The present estimates already give $a_{d,j}>0$ for $\lfloor d/4\rfloor<j\le d-2$ when $d\ge2^{62}$, and Theorem~\ref{thm:intro-tournament} gives $a_{d,d-1}>0$ for every $d\ge3$. The remaining task is to control the individual coordinates of small index and to obtain estimates valid below the eventual threshold. These are stronger requirements than the weighted-sum bounds proved here, and they isolate the gap between the observed finite pattern and a proof in all dimensions.

\subsection*{Data availability}
The exact data are available at
\begin{center}
\url{https://github.com/tashimir/ehrhart-132-213-polytopes}.
\end{center}
The single CSV described in Section~\ref{sec:finite} can be downloaded from the
\href{https://github.com/tashimir/ehrhart-132-213-polytopes/releases/tag/v1.0.0}{version 1.0.0 data release}.
It contains all $500500$
scaled coordinates for $1\le d\le1000$. Its three columns are
\texttt{d,j,alpha\_scaled}, with $\alpha_{d,j}=d!a_{d,j}$.
Algorithm~\ref{alg:finite-certificate} reconstructs every entry from the
displayed formulas and checks the $498500$ strict tail inequalities
using exact arithmetic.

\end{document}